\documentclass[times]{nlaauth}

\usepackage{moreverb}

\usepackage[colorlinks,bookmarksopen,bookmarksnumbered,citecolor=red,urlcolor=red]{hyperref}

\newcommand\BibTeX{{\rmfamily B\kern-.05em \textsc{i\kern-.025em b}\kern-.08em
T\kern-.1667em\lower.7ex\hbox{E}\kern-.125emX}}

\def\volumeyear{2017}

\usepackage[utf8]{inputenc} 		 					 	
\usepackage[english]{babel}                     	   		
\usepackage{amsmath,amsfonts,amssymb,amsxtra} 				
\usepackage{bbm}                    					 	
\usepackage{graphicx}               					 	
\usepackage{caption}										
\usepackage{subcaption}										
\usepackage{algpseudocode}									
\usepackage{hyperref}										
\usepackage{framed}											
\usepackage{mathtools}										
\usepackage{tikz}											
\usepackage{bm}												
\usepackage{upgreek}										
\usepackage{chngcntr}										
\usepackage{diagbox}										
\usepackage{array}

\newcolumntype{x}[1]{%
>{\centering\hspace{0pt}}p{#1}}%

\newtheorem{theorem}{Theorem}

\newtheorem{algorithm}[theorem]{Algorithm}
\newtheorem{lemma}[theorem]{Lemma}

\newcommand{\norm}[1]{\lVert#1\rVert} 							
\newcommand{\normb}[1]{\big\lVert#1\big\rVert} 					
\newcommand{\dotprod}[2]{\langle#1,#2\rangle} 					
\DeclareMathOperator{\Span}{span}								
\DeclareMathOperator{\blkdiag}{blkdiag}							
\DeclareMathOperator*{\diag}{diag}						
\DeclareMathOperator*{\rank}{rank}						
\newcommand{\R}{{\mathbb{R}}}       							
\newcommand{\N}{{\mathbb{N}}}       							

\makeatletter
\newcases{mycases}{\quad}{%
  \hfil$\m@th\displaystyle{##}$}{$\m@th\displaystyle{##}$\hfil}{\lbrace}{.}
\makeatother

\begin{document}

\runningheads{Gennadij~Heidel~and~Andy~Wathen}{Preconditioning for boundary control in fluid dynamics}

\title{Preconditioning for boundary control problems in incompressible fluid dynamics}

\renewcommand{\thefootnote}{\fnsymbol{footnote}}
\author{Gennadij~Heidel\renewcommand{\thefootnote}{\arabic{footnote}}\footnotemark[1]\renewcommand{\thefootnote}{\fnsymbol{footnote}}\corrauth~and~Andy~Wathen\renewcommand{\thefootnote}{\arabic{footnote}}\footnotemark[2]}

\address{\centering \renewcommand{\thefootnote}{\arabic{footnote}} \footnotemark[1]Fachbereich IV - Mathematik, Universit\"at Trier, 54286 Trier, Germany\\
\footnotemark[2]Mathematical Institute, University of Oxford, Radcliffe Observatory Quarter, Oxford, OX2 6GG, UK}

\corraddr{\href{mailto:heidel@uni-trier.de}{heidel@uni-trier.de}}

\begin{abstract}
PDE-constrained optimization problems arise in many physical applications, prominently in incompressible fluid dynamics. In recent
research, efficient solvers for optimization problems governed by the Stokes and Navier--Stokes equations have been developed which are
mostly designed for distributed control. Our work closes a gap by showing the effectiveness of an appropriately modified preconditioner to 
the case of Stokes boundary control. We also discuss the applicability of an analogous preconditioner for Navier--Stokes boundary control 
and provide some numerical results.
\end{abstract}


\maketitle

\section{Introduction} \label{sec:introduction}

\renewcommand{\thefootnote}{\arabic{footnote}}
\setcounter{footnote}{0}

Optimal control problems were first introduced by Lions \cite{lions1971optimal} and have recently attracted considerable research 
interest in applied mathematics, in terms of both theory \cite{troltzsch2010optimal} and computation \cite{borzi2012computational}. An
important area where problems of this kind naturally arise is that of fluid dynamics. An important problem within this field is that of
systems governed by the Navier--Stokes equations and their limiting case for very viscous flow, the Stokes equations.

For both steady-state Stokes and Navier--Stokes forward problems efficient solvers for finite element discretizations have been 
developed, see \cite{wathen1993fast,silvester1993fast} for the Stokes case and 
\cite{elman1999preconditioning,silvester2001efficient,kay2002preconditioner,elman2006block} for the Navier--Stokes case. These methods are
built on Krylov subspace iterations with specially designed preconditioners to achieve rapid convergence. An overview of state-of-the-art 
methods is given in the recent text by Elman et al. \cite{elman2014finite}, for a survey of preconditioning in general see
\cite{wathen2015preconditioning}.

More recently, solvers for control problems governed by the Stokes and Navier--Stokes equations based on these forward solvers have been
developed. Our work is largely based on a preconditioner developed by Rees and Wathen \cite{rees2011preconditioning} for distributed control
problems of the Stokes equations. A new development in PDE-constrained optimization is parameter-robust preconditioning, i.\,e.
methods whose quality does not depend on regularization parameters in the cost function. The notion was introduced by Zulehner et al. \cite{schoberl2007symmetric,schoberl2011robust}; more recently, Pearson and Wathen have developed remarkable results for distributed 
Poisson control \cite{pearson2012new}. This framework has been extended to Poisson boundary control and heat control
\cite{pearson2012regularization}, distributed Stokes control in the steady-state \cite{pearson2015development} and the time-dependent
case \cite{stoll2013allatonce}, and distributed steady-state Navier--Stokes control \cite{pearson2015preconditioned}.

This paper is organized as follows. In Section~\ref{sec:stokesproblem} we introduce the Stokes boundary control problem and
discuss its discretization and optimality conditions. In Section~\ref{sec:stokesprecond} we present a Rees--Wathen type preconditioner
for the optimality system of this problem. In Section~\ref{sec:navierproblem} we introduce the analogous Navier--Stokes boundary control
problem, discuss the nonliner iteration employed and provide the discretization and optimality conditions for the linearized problems.
In Section~\ref{sec:navierprecond} we discuss spectral properties of the linearized and discretized Navier--Stokes problem and develop
a Rees--Wathen type preconditioner for this problem. In Section~\ref{sec:numres} we present numerical results to highlight the performance
of our preconditioners, and in Section~\ref{sec:conclusions} we make some concluding remarks and discuss possible extensions of this work.

\section{The Stokes boundary control problem} \label{sec:stokesproblem}

Let $\varOmega$ be a channel domain in $\R^2$ or $\R^3$, and let $\widehat{\vec{v}}$ and $\widehat{p}$ be functions on $\varOmega$ that define
a desired velocity and pressure profile. We want to manipulate the inflow of the channel $\partial\varOmega_{\text{in}}$ in such a way that 
the Stokes flow profile is as close as possible to $(\widehat{\vec{v}},\widehat{p})$; in a mathematical sense, applying a force to the 
boundary of the channel is the same as imposing Neumann boundary conditions, this gives us a Neumann control problem in a natural way.

This may be formulated as minimizing a least-squares cost functional subject to the Stokes equations as the constraint, i.\,e.
\begin{equation} \label{eqn:stokesproblem}
	\begin{split}
		\min_{\vec{v},\,p,\,\vec{u}}\,
		& \frac{1}{2} \normb{\vec{v} - \widehat{\vec{v}}}_{L^2(\varOmega)^2}^2
		+ \frac{\alpha}{2} \normb{p - \widehat{p}}_{L^2(\varOmega)}^2
		+ \frac{\beta}{2} \normb{\vec{u}}_{L^2(\partial \varOmega_{\text{in}})^2}^2 \\
		& \text{such that}
		\begin{mycases}
			- \nabla^2 \vec{v} + \nabla p & = \vec{0} \quad \text{in} \;\, \varOmega, \\
			\nabla \cdot \vec{v} & = 0 \quad \text{in} \;\, \varOmega, \\
			\vec{v} &= \vec{0} \quad \text{on} \;\,  \partial\varOmega_D, \\
			\tfrac{\partial\vec{v}}{\partial n} - p \vec{n} & = \vec{u} \quad \text{on} \;\,  \partial\varOmega_{\text{in}}, \\
			\tfrac{\partial\vec{v}}{\partial n} - p \vec{n} & = \vec{0} \quad \text{on} \;\,  \partial\varOmega_{\text{out}}.
		\end{mycases}
	\end{split}
\end{equation}
\noindent Here the Dirichlet boundary $\varOmega_D$ represents the walls of the channel where the flow is equal to zero, and a zero-stress
boundary condition is chosen on the outflow $\partial\varOmega_{\text{out}}$; this ensures that the fluid leaves the channel domain freely
without a force applied to it. The desired pressure is typically normalized to ${\widehat{p} \equiv 0}$. The positive regularization 
parameters $\alpha$ and $\beta$ are chosen a priori; as long as they are not too large, the main focus of the cost function lies on the 
velocity term which characterizes the desired flow profile. The outward normal unit vector is denoted by $\vec{n}$, as usual.

Problem formulations with a tracking-type objective function as in \eqref{eqn:stokesproblem} are an important class of optimal control problems
in the literature, see \cite{troltzsch2010optimal}. Other important classes of boundary control problems in fluid dynamics, which are beyond the
scope of this paper, are minimum vorticity control and Dirichlet boundary control, e.\,g. \cite{heinkenschloss1998formulation}.

There are two different approaches to the discretization of this problem. We can either discretize first and then find optimality
conditions for the discretized system; or we can find optimality conditions for the infinite-dimensional problem \eqref{eqn:stokesproblem} 
using the formal Lagrange technique \cite{lions1971optimal,troltzsch2010optimal}, and then discretize the obtained equations. Since the 
Stokes equations are selfadjoint, both approaches lead to the same discrete optimality conditions, as long as we use an adjoint-consistent
discretization. Therefore, we only consider the discretize-then-optimize approach here.

Let $\{\vec{\varphi}_j\}_{j=1}^{n_v+n_\partial}$ and $\{\psi_k\}_{k=1}^{n_p}$ be finite element bases that form a stable mixed finite 
element discretization for the Stokes equations---see, for example, \cite[Chapter~3]{elman2014finite} for further details. Note that, in 
general, we would also need basis functions $\vec{\varphi}_{n_v+1},\cdots,\vec{\varphi}_{n_v+n_\partial}$ which interpolate the Dirichlet
boundary data; this need not be considered here since the Dirichlet boundary data in \eqref{eqn:stokesproblem} is identically zero. Since the 
control $\vec{u}$ is another unknown, we also choose a finite basis element basis $\{\vec{\chi}_l\}_{l=1}^{n_u}$ which lives on the inflow boundary
of the channel. Since the boundary of a $d$-dimensional domain $\varOmega$ is a manifold of dimension $d-1$, the control space is canonically 
isomorphic to a finite element space on a domain of dimension $d-1$; we will not distinguish between this and a boundary finite element 
space. Let $\vec{v}_h = \sum_{j=1}^{n_v} \bm{v}_j \vec{\varphi}_j$,
$p_h = \sum_{k=1}^{n_p} \bm{p}_k \psi_k$ and $\vec{u}_h = \sum_{l=1}^{n_u} \bm{u}_l \vec{\chi}_l$ be finite-dimensional approximations
to $\vec{v}$, $p$ and $\vec{u}$, respectively. Then the discretized Stokes equations are are given by
\begin{equation} \label{eqn:discrstokes}
	\begin{bmatrix}
		\bm{A} & B^T \\
		B & O
	\end{bmatrix}
	\begin{pmatrix}
		\mathbf{v} \\
		\mathbf{p}
	\end{pmatrix} =
	\begin{pmatrix}
		\widehat{\bm{Q}}\mathbf{u} \\
		\mathbf{0}
	\end{pmatrix},
\end{equation}
\noindent where $\mathbf{v}$, $\mathbf{p}$ and $\mathbf{u}$ are the coefficient vectors in the expansions of $\vec{v}_h$, $p_h$ and 
$\vec{u}_h$, respectively, and the matrices are given by $\bm{A}=[\int_{\varOmega}\nabla\vec{\varphi}_j:\nabla \vec{\varphi}_i]$,
${B=[\int_{\varOmega}\psi_k\nabla\cdot\vec{\varphi}_j]}$ and
$\widehat{\bm{Q}}=[\int_{\partial\varOmega_{\text{in}}}\vec{\varphi}_i\cdot\vec{\chi}_l]$. Note that we use the convention to denote
Gramian matrices obtained from vector-valued functions by bold letters.

The discretized cost functional of \eqref{eqn:stokesproblem} is given by
\begin{equation} \label{eqn:discrcost}
	\min_{\mathbf{v},\,\mathbf{p},\,\mathbf{u}}\,
	\frac{1}{2} \dotprod{\bm{Q}_{\vec{v}}\mathbf{v}}{\mathbf{v}} - \dotprod{\mathbf{b}}{\mathbf{v}}
	+ \frac{\alpha}{2} \dotprod{Q_p\mathbf{p}}{\mathbf{p}} - \alpha\dotprod{\mathbf{d}}{\mathbf{p}}
	+ \frac{\beta}{2}\dotprod{\bm{Q}_{\vec{u}}\mathbf{u}}{\mathbf{u}},
\end{equation}
\noindent where the mass matrices are given by $\bm{Q}_{\vec{v}}=[\int_{\varOmega}\vec{\varphi}_i\cdot\vec{\varphi}_j]$,
$Q_p = [\int_{\varOmega} \psi_i \psi_j]$ and $\bm{Q}_{\vec{u}}=[\int_{\partial\varOmega_{\text{in}}}\vec{\chi}_i\cdot\vec{\chi}_j]$, and the 
vectors $\mathbf{b}=(\int_\varOmega\vec{\varphi}_j\cdot\widehat{\vec{v}})$ and $\mathbf{d}=(\int_\varOmega \psi_k \cdot \widehat{p})$ 
contain the linear terms.

In practice, it is convenient to choose the bases $\{\vec{\varphi}_j\}$ for $\vec{v}_h$ and $\{\vec{\chi}_l\}$ for $\vec{u}_h$
such that for every $l$ there exists some $j(l)$ with ${\vec{\chi}_l=\vec{\varphi}_{j(l)}|_{\partial\varOmega_{\text{in}}}}$. Then the 
$j^{\text{\tiny th}}$ row of the mixed mass matrix $\widehat{\bm{Q}}$ will be equal to the $l^{\text{\tiny th}}$ row of $\bm{Q}_{\vec{u}}$
if $j=j(l)$ for some $l$, and zero otherwise. We will call this control discretization control-consistent. Throughout this paper, we will 
use a control-consistent Taylor--Hood approximation \cite{taylor1973numerical}, i.\,e. a $\pmb{Q}_2$ approximation for the velocity (and hence for 
the control) and a $\pmb{Q}_1$ approximation for the pressure. 

As is usual, we approximate all velocity space components using a single scalar finite element space, which is given by a basis 
$\{\varphi_j\}$. Then, in two dimensions ${\bm{Q}_{\vec{v}}=\blkdiag(Q_v,Q_v)}$, where ${Q_v=[\int_{\varOmega} \varphi_i \varphi_j]}$, and 
analogously ${\bm{Q}_{\vec{u}}=\blkdiag(Q_u,Q_u)}$ and ${\bm{A}=\blkdiag(A,A)}$. The extension to three dimensions is obvious.

If we introduce adjoint variables $\bm{\uplambda}$ and $\bm{\upmu}$, then the KKT conditions for the discretized optimization problem are
given by
\begin{equation} \label{eqn:stokeskkt}
	\begin{bmatrix}
		\bm{Q}_{\vec{v}} & O & O & \bm{A} & B^T \\
		O & \alpha Q_p & O & B & O \\
		O & O & \beta\bm{Q}_{\vec{u}} & -\widehat{\bm{Q}}^T & O \\
		\bm{A} & B^T & -\widehat{\bm{Q}} & O & O \\
		B & O & O & O & O
	\end{bmatrix}
	\begin{pmatrix}
		\mathbf{v} \\
		\mathbf{p} \\
		\mathbf{u} \\
		\bm{\uplambda} \\
		\bm{\upmu}
	\end{pmatrix} =
	\begin{pmatrix}
		\mathbf{b} \\
		\alpha\mathbf{d} \\
		\mathbf{0} \\
		\mathbf{0} \\
		\mathbf{0} \\
	\end{pmatrix}.
\end{equation}
\noindent Note that the discrete cost function \eqref{eqn:discrcost} is strictly convex, thus the solution of \eqref{eqn:stokeskkt} is
guaranteed to be the global minimizer.

A well-known analytic solution of the Stokes equations on a channel domain is the Poiseuille flow, see \cite[pp.\,122ff]{elman2014finite}. 
On a simple square domain $\varOmega=(-1,1)^2$ where the left-hand boundary $\partial\varOmega_{\text{in}} = \{-1\}\times(-1,1)$ is
the inflow and the right-hand boundary $\partial\varOmega_{\text{out}} = \{1\}\times(-1,1)$ is the outflow, the velocity solution is given 
by $\vec{v} = (1-y^2,0)^T$. This describes a straight horizontal movement with a parabolic flow profile whose maximum is in the middle of
the channel. In our model problem we want to restrict this flow profile to the upper half of the channel 
with no fluid movement in the lower half, i.\,e. the desired velocity is given by
\begin{equation} \label{eqn:vhat}
	\widehat{v}_x =
	\begin{cases}
	4y - 4y^2 & \text{if } 0 \leq y < 1,\\
	0 & \text{if } -1 < y < 0,
	\end{cases}
	\qquad \widehat{v}_y = 0.
\end{equation}

\section{Preconditioning for the Stokes problem} \label{sec:stokesprecond}

The KKT matrix in \eqref{eqn:stokeskkt} is clearly of block saddle point structure in the form
\begin{equation*}
	\begin{bmatrix}
		\mathcal{A} & \mathcal{B}^T \\
		\mathcal{B} & \mathcal{O}
	\end{bmatrix}
\end{equation*}
\noindent with the blocks $\mathcal{A} = \blkdiag(\bm{Q}_{\vec{v}},Q_p,\bm{Q}_{\vec{u}})$ and 
\begin{equation*}
	\mathcal{B} = \begin{bmatrix} 
	\bm{A}&B^T&-\widehat{Q} \\
	B&O&O
	\end{bmatrix}.
\end{equation*}
For a comprehensive survey of 
numerical methods for such matrices see \cite{benzi2005numerical}.

We would like to work with preconditioned Krylov subspace methods. A well-known property of saddle point matrices is their indefiniteness, 
therefore, the method of choice is the MINRES iteration of Paige and Saunders \cite{paige1975solution}. The crucial part of an efficient 
method is the right choice of the preconditioner. A block diagonal preconditioner is given by
\begin{equation} \label{eqn:bd}
	\begin{bmatrix}
		\mathcal{A} & \mathcal{O} \\
		\mathcal{O} & \mathcal{S}
	\end{bmatrix},
\end{equation}
\noindent where $\mathcal{S}=\mathcal{B}\mathcal{A}^{-1}\mathcal{B}^T$ is the Schur complement. This preconditioner is ideal in the sense
that the preconditioned system will have only three distinct eigenvalues, namely $1$ and $(1\pm\sqrt{5})/2$, see \cite{murphy2000note}.
However, the practical application of this preconditioner requires the solution of systems with $\mathcal{A}$ and $\mathcal{S}$ which can be
expected to be prohibitively expensive, therefore positive definite approximations $\widetilde{\mathcal{A}} \approx \mathcal{A}$ and 
$\widetilde{\mathcal{S}} \approx \mathcal{S}$ are used. This gives good clustering of the eigenvalues as long as the approximations are 
spectrally close to the exact operators, see \cite[Theorem~4.7]{elman2014finite} for a rigorous statement.

Here we describe approximations for $\mathcal{A}$ and $\mathcal{S}$ developed by Rees and Wathen \cite{rees2011preconditioning} for distributed
Stokes control and justify their usefulness for boundary control.

\subsection{Approximation of the (1,1) block} \label{subsec:11block}

Since the (1,1) block has a block-diagonal structure itself, it is sufficient to approximate all of the three blocks separately. A result 
for a general mass matrix $Q$ due to Wathen \cite{wathen1987realistic} says: for $D=\diag(Q)$, the eigenvalues of $D^{-1}Q$ are bounded 
below and above by some constants $\theta$ and $\varTheta$, independent of mesh size, and these constants can be calculated explicitly as minimal and maximal 
eigenvalues of diagonally scaled element mass matrices. For a control-consistent Taylor--Hood approximation they are given in 
Table~\ref{tab:massbounds}.

\begin{table}[ht]
	\centering
	\caption{Eigenvalue bounds for diagonally scaled mass matrices for a $\pmb{Q}_2$-$\pmb{Q}_1$ approximation (valid for any domain).} 
	\label{tab:massbounds}
	\begin{tabular}{|c||c|c|} \hline
 		  & $\theta$  & $\varTheta$  \\ \hline\hline
  		$\bm{Q}_{\vec{v}}$ & $1/4$ & $25/16$ \\ \hline
  		$Q_p$   & $1/4$ & $9/4$ \\ \hline
  		$\bm{Q}_{\vec{u}}$        & $1/2$  & $5/4$ \\ \hline
	\end{tabular}
\end{table}

Therefore, diagonal scaling could be used as a good preconditioner for $\mathcal{A}$. An even better result can by achieved by using a fixed 
number of steps of a Chebyshev semi-iteration \cite{golub1961chebyshev}, which accelerates the convergence of a simple splitting method by 
substituting an iterate by a linear combination of all previous iterates. The optimal linear combination can be found from the eigenvalue 
bounds in Table~\ref{tab:massbounds} and the minimax property of the Chebyshev polynomials. The usefulness of the Chebyshev semi-iteration 
for problems with the mass matrix was shown by Wathen and Rees \cite{wathen2009chebyshev}.

\begin{table}[t]
	\centering
	\caption{Minimal and maximal eigenvalues and condition numbers of $M_{\text{C}}^{-1}Q$ with $20$ iterations for various mass matrices.}
	\label{tab:masscheb}
	\begin{tabular}{|c||c|c|c|} \hline
 		 & $\lambda_{\min}$  & $\lambda_{\max}$ & $\kappa$ \\ \hline\hline
  		$\bm{Q}_{\vec{v}}$ & $0.999999912603445$ & $1.000000087211418$ & $1.000000235482498$ \\ \hline
  		$Q_p$ & $0.999998092651363$ & $1.000001906960303$ & $1.000004093538232$ \\ \hline
  		$\bm{Q}_{\vec{u}}$ & $0.999999999999775$  & $1.000000000000209$ & $1.000000000000467$ \\ \hline
	\end{tabular}
\end{table}

To quantify this, we calculate the eigenvalues of $M_\text{C}^{-1}Q$, where $M_\text{C}$ is the preconditioner given by $20$ Chebyshev 
steps with initial guess equal to zero, see Table~\ref{tab:masscheb}. It shows that $20$ steps are enough to get very good approximation for 
all mass matrices. For more numerical results, see Rees and Stoll \cite[Table~I]{rees2010block}.

\subsection{Approximation of the Schur complement} \label{subsec:schur}

Now we want an approximation for the Schur complement
\begin{equation*}
	\begin{split}
		\mathcal{S} & = \mathcal{B} \mathcal{A}^{-1} \mathcal{B}^T = 
		\begin{bmatrix}
			\bm{A} & B^T & -\widehat{Q} \\
			B & O & O
		\end{bmatrix}
		\begin{bmatrix}
			\bm{Q}_{\vec{v}}^{-1} & O & O \\
			O & \tfrac{1}{\alpha} Q_p^{-1} & O \\
			O & O & \tfrac{1}{\beta} \bm{Q}_{\vec{u}}^{-1}
		\end{bmatrix}
		\begin{bmatrix}
			\bm{A} & B^T \\
			B & O \\
			-\widehat{Q}^T & O
		\end{bmatrix} \\
		& = \underbrace{\begin{bmatrix}
			\bm{A} & B^T\\
			B & O
		\end{bmatrix}}_{\eqqcolon \mathcal{K}}
		\underbrace{\begin{bmatrix}
			\bm{Q}_{\vec{v}}^{-1} & O \\
			O & \tfrac{1}{\alpha} Q_p^{-1} \\
		\end{bmatrix}}_{\eqqcolon \mathcal{Q}^{-1}}
		\begin{bmatrix}
			\bm{A} & B^T \\
			B & O \\
		\end{bmatrix}
		+ \tfrac{1}{\beta} \underbrace{\begin{bmatrix}
			\widehat{Q} \bm{Q}_{\vec{u}}^{-1} \widehat{Q}^T & O \\
			O & O \\
		\end{bmatrix}}_{\eqqcolon \mathcal{L}}.
	\end{split}
\end{equation*}
\noindent So, $\mathcal{S} = \mathcal{K}\mathcal{Q}^{-1}\mathcal{K}+\tfrac{1}{\beta}\mathcal{L}$. Note that the matrix $\mathcal{K}$ is
the same Stokes operator as in \eqref{sec:stokesproblem}. This additive structure of the Schur complement makes exact precondtitioning 
difficult. For problems of this kind, Rees, Dollar and Wathen \cite[Corollary~3.3]{rees2010optimal} suggest dropping the second term---and 
hence using the approximation $\widetilde{\mathcal{S}} = \mathcal{K}\mathcal{Q}^{-1}\mathcal{K}$---and present eigenvalue bounds
(depending on $\beta$) for a control problem governed by the Poisson equation. This is also the strategy used in the Rees--Wathen 
preconditioner. The intuitive reasoning behind this is that the first summand  clearly carries more information in some sense---it contains 
the discrete Stokes operator whereas the second summand consists only of mass matrices which can be thought of as indentity or natural 
inclusion operators in some finite element spaces. Therefore, if $\beta$ is sufficently large (hence $1/\beta$ sufficiently small) one can 
hope that this gives a reasonable approximation. In \cite{rees2011preconditioning} Rees and Wathen show that this strategy applied to 
distributed control gives good results with respect to the grid size for $\beta \geq 10^{-4}$.

Ideally, we would like to find a regularization-robust preconditioner for our optimal control problem, i.\,e. the quality of preconditioning 
should be independent of the regularization parameter $\beta$. Pearson and Wathen \cite{pearson2012new} have proved the existence of such 
preconditioners for distributed Poisson control problems. In \cite{pearson2015development} a preconditioner for distributed control of the 
Stokes equations is presented which shows regularization-robust behaviour in numerics. Unfortunately, the framework of the Pearson--Wathen 
preconditioner and of the derived preconditioner for the Stokes equations heavily relies on the fact that in the case of distributed control, 
the control $\mathbf{u}$ is just a scalar multiple of the adjoint $\bm{\uplambda}$, and thus, the optimality system can be reduced to a $4 
\times 4$ block structure. This is clearly not applicable to our situation, as $\mathbf{u}$ and $\bm{\uplambda}$ do not even have the same 
dimension!

By construction, the Rees--Wathen preconditioner for Stokes control cannot be regularization-robust. However, in our case it can be proved
to be ``almost regularization-robust'' in the sense that $\rank(\tfrac{1}{\beta}\mathcal{L}) = n_u \ll n_v+n_p = \rank(\mathcal{S})$. Thus,
all we lose by our choice of approximation is a low-rank perturbation. The analysis of symmetric rank-$1$-perturbations is due to 
Wilkinson \cite[pp.\,87ff]{wilkinson1965algebraic} and can be written in the form of the following lemma (here and in the 
rest of the paper we assume the usual ordering $\lambda_n \leq \dotsb \leq \lambda_1$ for the eigenvalues of a symmetric
$n \times n$ matrix).

\begin{lemma} \label{lem:lowrank}
	\cite[Theorem~8.1.8]{golub2013matrix} Suppose $B = A +\tau \mathbf{c}\mathbf{c}^T$ where $A \in \mathbb{R}^{n \times n}$ is symmetric 
	and $\tau \geq 0$. Then
	\begin{equation*}
		\lambda_i(B) \in [\lambda_i(A),\lambda_{i-1}(A)], \quad \text{for} \;\, i = 2,\dots,n.
	\end{equation*}
\end{lemma}

By an inductive argument, this can be generalized to additive perturbations of arbitrary rank.

\begin{theorem} \label{thm:lowrank}
	Suppose $B = A + L$ where $A,\,L \in \mathbb{R}^{n \times n}$ are symmetric, $L$ is positive semidefinite, and 
	$\mathrm{rank}(L) \leq m < n$. Then
	\begin{equation*}
		\lambda_i(B) \in [\lambda_i(A),\lambda_{i-m}(A)], \quad \text{for} \;\, i = m+1,\dots,n.
	\end{equation*}
\end{theorem}

\begin{proof}
	For $m=0$ there is nothing to prove. Let the statement be true for some $m \in \N_0$ and $L$ be a symmetric positive semidefinite 
	matrix with $\rank(L) \leq m+1$. Then $L=K+\tau_{m+1}\mathbf{c}_{m+1}\mathbf{c}_{m+1}^T$ can be written as 
	\begin{equation*}
		L = \underbrace{\sum_{j=1}^m\tau_j\mathbf{c}_j\mathbf{c}_j^T}_{\eqqcolon K} + \tau_{m+1}\mathbf{c}_{m+1}\mathbf{c}_{m+1}^T
	\end{equation*}
	with $\tau_{m+1} \geq 0$. Then $\rank(K) \leq m$ and, by the induction hypothesis,
	\begin{align*}
		\lambda_i(A+K) & \in [\lambda_i(A),\lambda_{i-m}(A)], \quad \text{for} \;\, i = m+1,\dots,n,
		\intertext{and by Lemma~\ref{lem:lowrank},}
		\lambda_i(B) & \in [\lambda_i(A+K),\lambda_{i-1}(A+K)], \quad \text{for} \;\, i = 2,\dots,n.
		\intertext{Combining these, we get}
		\lambda_i(B) & \in [\lambda_i(A),\lambda_{i-m-1}(A)], \quad \text{for} \;\, i = m+2,\dots,n,
	\end{align*}
	as desired.
\end{proof}

Therefore, if we drop the low-rank perturbation, we get a parameter-robust approximation for all but $n_u$ eigenvalues. This is illustrated
in Figure~\ref{fig:paramrobust}, where $n_v = 162$, $n_p = 25$ and $n_u = 18$. The red line indicates ${\rank(\mathcal{K})-\rank{\mathcal{L}}}={n_v+n_p-n_u}$; we observe a good clustering of 
all but the largest $n_u$ eigenvalues around $1$ even for a fairly small $\beta$, the non-clustered eigenvalues show a behaviour dependent 
on $\beta$.

\begin{figure}[t]
	\centering
	\begin{tabular} {cc}
		\includegraphics[scale=0.125]{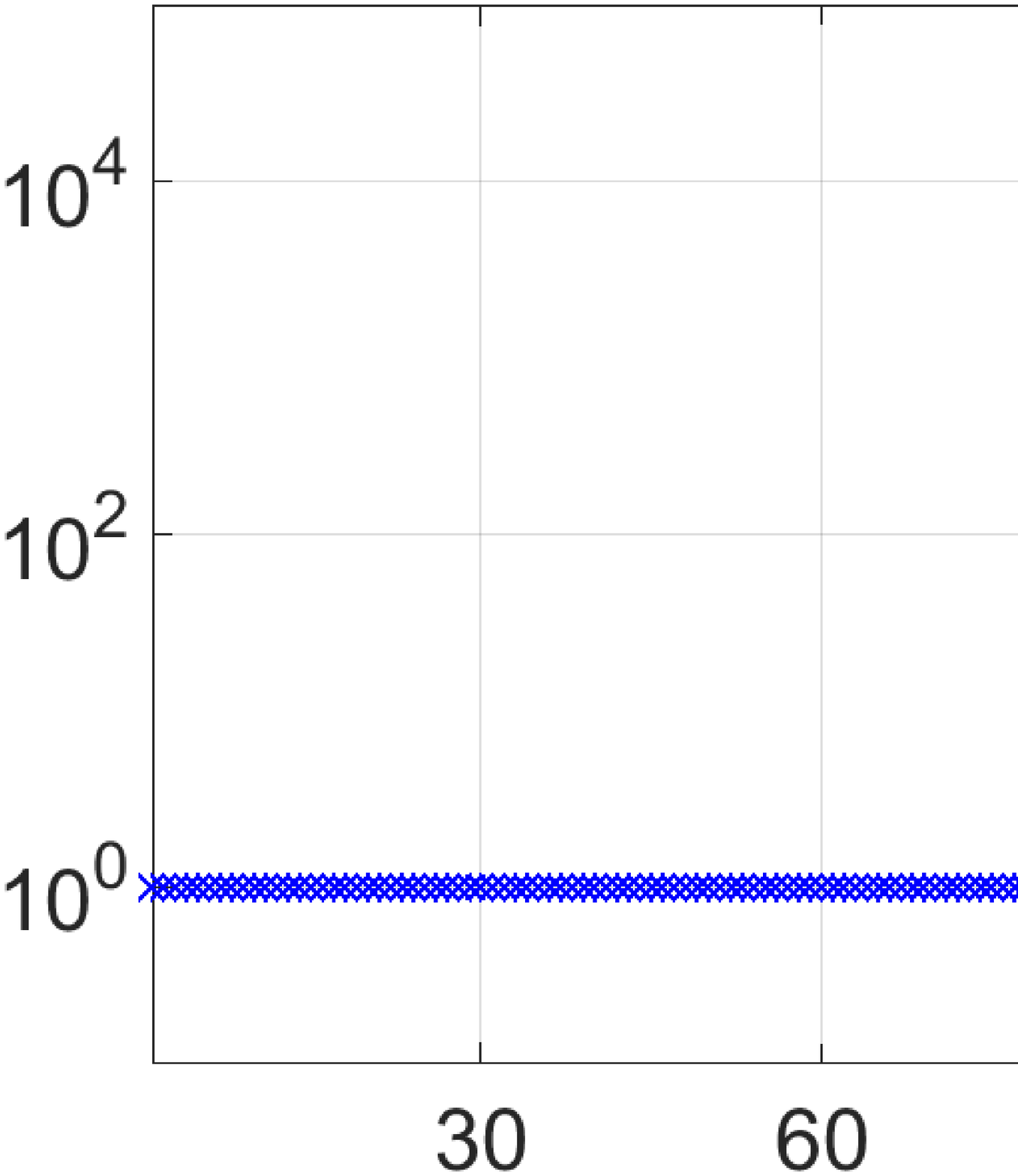} &
		\includegraphics[scale=0.125]{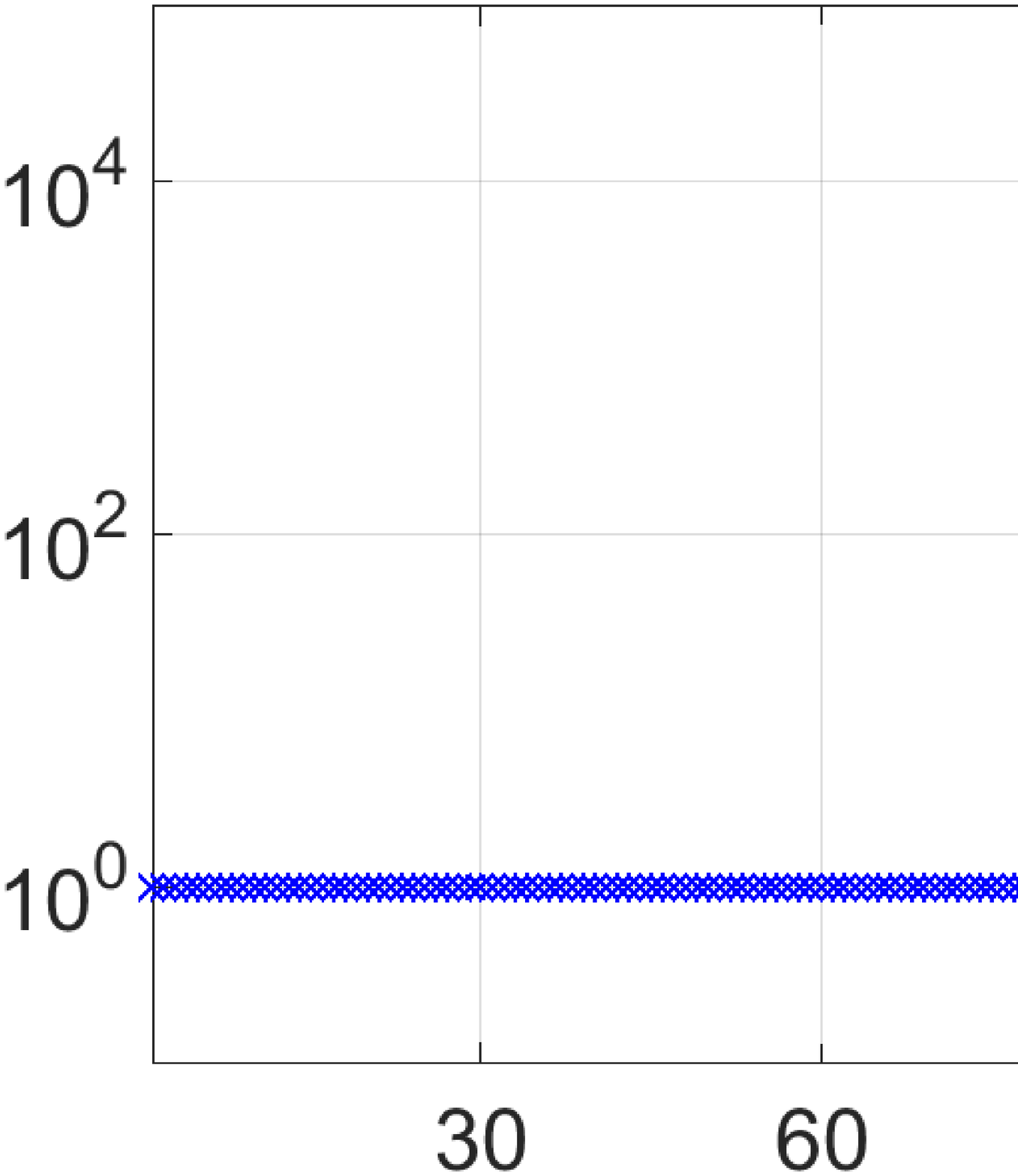} \\
	\end{tabular}
	\caption{Eigenvalues of $\widetilde{\mathcal{S}}^{-1}\mathcal{S}$ for $\beta=10^{-3}$ (left) and $\beta=10^{-6}$ (right).}
	\label{fig:paramrobust}
\end{figure}

Now we can hope that we will get good results if we can find an approximation for the action of
$\widetilde{\mathcal{S}}^{-1} = \mathcal{K}^{-1}\mathcal{Q}\mathcal{K}^{-1}$. This needs two inverses of the discrete Stokes operator
$\mathcal{K}$ and one multiply by $\mathcal{Q}$; the latter is obviously trivial. The approximation of the discrete Stokes operator requires
some extra care here, since we cannot simply use a well-known Stokes preconditioner such as \cite{silvester1993fast}. The issue that has 
been pointed out by Braess and Peisker \cite{braess1986numerical} is the following: if we have a good preconditioner $\widetilde{K}$ for a 
matrix $K$, then $\widetilde{K}^T\widetilde{K}$ is not necessarily a good preconditioner for $K^TK$. Indeed, we find that commonly used 
preconditioners for the Stokes operator fail in our situation. In their paper, Braess and Peisker also show sufficient conditions for a
squared preconditioner to work.

\begin{theorem}[Braess--Peisker conditions] \label{thm:braesspeisker}
	We consider the linear equation ${K\mathbf{x} = \mathbf{b}}$. Let $\widetilde{K}_j$ be a sequence of invertible matrices, such that 
	$\mathbf{x}^{(j)} \coloneqq \widetilde{K}_j^{-1} \mathbf{b}$ converges to the exact solution $\mathbf{x}$ in the sense that
	\begin{equation*}
		\normb{\mathbf{x}^{(j)} - \mathbf{x}}_2 \leq \eta_j \norm{\mathbf{x}}_2
	\end{equation*}
	with $\eta_j \rightarrow 0$. \\
	Then, for $j$ large enough, all $\mathbf{y}\neq\mathbf{0}$ satisfy
	\begin{equation*}
		(1 - \eta_j)^2 \leq
		\frac{\dotprod{KK^T\mathbf{y}}{\mathbf{y}}}{\dotprod{\widetilde{K}_j\widetilde{K}_j^T\mathbf{y}}{\mathbf{y}}} 
		\leq (1 + \eta_j)^2.
	\end{equation*}
\end{theorem}

Note that the transposition of $\mathcal{K}$ is not strictly needed here since the discrete Stokes operator is symmetric; but in the 
Navier--Stokes case we will have to deal with a nonsymmetric $\mathcal{K}$ in the next chapter; therefore we include this case here.

Remember that we actually need to precondition $\mathcal{K}^T \mathcal{Q}^{-1} \mathcal{K}$. Rees and Wathen \cite{rees2011preconditioning}
show that the matrix $\mathcal{Q}^{-1}$ simply introduces a scaling to the Braess--Peisker result in the form
\begin{equation*}
	c_* (1 - \eta_j)^2 \leq
	\frac{\dotprod{KK^T\mathbf{x}}{\mathbf{x}}}{\dotprod{\widetilde{K}_j\widetilde{K}_j^T\mathbf{x}}{\mathbf{x}}} 
	C_* \leq (1 + \eta_j)^2,
\end{equation*}
\noindent with some constants $0 < c_* \leq 1$ and $1 \leq C_*$.

By Theorem~\ref{thm:braesspeisker} we have to approximate the Stokes equations by a contracting linear iteration. The Rees--Wathen 
preconditioner uses the inexact Uzawa \cite{uzawa1958iterative} iteration given by Algorithm~\ref{alg:uzawa} for a generic saddle point 
problem

\begin{equation*}
	\begin{bmatrix}
		\bm{A} & B^T \\
		B & O
	\end{bmatrix}
	\begin{pmatrix}
		\mathbf{v} \\
		\mathbf{p}
	\end{pmatrix}
	=
	\begin{pmatrix}
		\mathbf{f} \\
		\mathbf{g}
	\end{pmatrix},
\end{equation*}

\noindent with preconditioners $\widetilde{\bm{A}}$ for $\bm{A} $ and $\widetilde{S}$ for the Schur complement $S = B\bm{A}^{-1}B^T$. The 
convergence properties of the inexact Uzawa method have been studied in \cite{elman1994inexact,bramble1997analysis,zulehner2002analysis}. It 
can be shown to provide preconditioning for the Stokes operator independently of the grid size as long as the approximations 
$\widetilde{\bm{A}}$ and $\widetilde{S}$ are good enough, see \cite[Subsection~2.4]{rees2011preconditioning}. For the Laplacian operator
$\bm{A}$ a fixed number of multigrid cycles \cite{hackbusch1985multigrid,trottenberg2001multigrid} may be used; this is known to give a
spectrally equivalent preconditioner, that is if the action of a fixed number of multigrid cycles is given by a matrix $\widetilde{\bm{A}}$, 
then bounds $\delta$ and $\varDelta$ independent of the grid size $h$ exist such that ${0 < \delta \leq 
\lambda(\widetilde{\bm{A}}^{-1}\bm{A}) \leq \varDelta}$, see \cite[Section~2.5]{elman2014finite}. The Schur complement of the discrete Stokes
operator is spectrally equivalent to the pressure mass matrix $Q_p$, see \cite[Theorem~3.29]{elman2014finite}; hence a fixed number of
Chebyshev steps can be used here, as discussed in the previous subsection.

\begin{algorithm}[Inexact Uzawa] \label{alg:uzawa}
	\rm
	\begin{algorithmic}[0]
		\State Choose $\sigma, \, \tau > 0$
		\State Choose $\mathbf{v}^{(0)}$, $\mathbf{p}^{(0)}$
		\For{$k=0$ \textbf{until} convergence}
			\State Solve $\widetilde{\bm{A}}\bm{\updelta}\mathbf{v}^{(k)}
			= \mathbf{f} - \bm{A}\mathbf{v}^{(k)} - B^T \mathbf{p}^{(k)}$
			\State $\mathbf{v}^{(k+1)} = \mathbf{v}^{(k)} + \sigma\bm{\updelta}\mathbf{v}^{(k)}$
			\State Solve $\tfrac{1}{\tau}\widetilde{S}\bm{\updelta}\mathbf{p}^{(k)} = B\mathbf{v}^{(k+1)} - \mathbf{g}$
			\State $\mathbf{p}^{(k+1)} = \mathbf{p}^{(k)} + \tau\bm{\updelta}\mathbf{p}^{(k)}$
		\EndFor
	\end{algorithmic}
\end{algorithm}

\section{The Navier--Stokes boundary control problem} \label{sec:navierproblem}

We would like to extend our treatment to the more general case of Navier--Stokes control. The analogue to the Stokes problem is given by
\begin{equation} \label{eqn:navierproblem}
	\begin{split}
		\min_{\vec{v},\,p,\,\vec{u}}\,
		& \frac{1}{2} \normb{\vec{v} - \widehat{\vec{v}}}_{L^2(\varOmega)^2}^2
		+ \frac{\alpha}{2} \normb{p - \widehat{p}}_{L^2(\varOmega)}^2
		+ \frac{\beta}{2} \normb{\vec{u}}_{L^2(\partial \varOmega_{\text{in}})^2}^2 \\
		& \text{such that}
		\begin{mycases}
			- \nu \nabla^2 \vec{v} + \vec{v} \cdot \nabla \vec{v} + \nabla p & = \vec{0}
			\quad \text{in} \;\, \varOmega, \\
			\nabla \cdot \vec{v} & = 0 \quad \text{in} \;\, \varOmega, \\
			\vec{v} &= \vec{0} \quad \text{on} \;\,  \partial\varOmega_D, \\
			\nu \tfrac{\partial\vec{v}}{\partial n} - p \vec{n} & = \vec{u} \quad \text{on} \;\,  \partial\varOmega_{\text{in}}, \\
			\nu \tfrac{\partial\vec{v}}{\partial n} - p \vec{n} & = \vec{0} \quad \text{on} \;\,  \partial\varOmega_{\text{out}},
		\end{mycases}
	\end{split}
\end{equation}
\noindent where the only new terms are the viscosity parameter $\nu$ and the nonlinear convection term $\vec{v}\cdot\nabla\vec{v}$. A 
Navier---Stokes flow is usually characterized by the Reynolds number $\mathcal{R}\sim 1/\nu$; in fact, for our channel domain 
$\mathcal{R}=1/\nu$. To solve \eqref{eqn:navierproblem} we need to linearize the constraint, i.\,e. the convection term. This involves 
computing solutions $(\vec{v}_k,p_k)$ to a sequence of linearized problems starting from some initial guess\footnote{We will choose 
$(\vec{v}_0,p_0)$ to be the solution of the corresponding Stokes problem \eqref{eqn:stokesproblem}.} $(\vec{v}_0,p_0)$. The simplest way to 
do this is by a fixed point iteration where we replace the nonlinear convection term $\vec{v}\cdot\nabla\vec{v}$ by its linearized version 
$\vec{v}_h\cdot\nabla\vec{v}$ with $\vec{v}_h$ being the velocity solution from the previous iterate. This is referred to as the Picard 
linearization of the Navier--Stokes equations and the resulting linear PDE is called the Oseen equation. Hence to find a solution of 
\eqref{eqn:navierproblem} we solve a sequence of Oseen control problems
\begin{equation} \label{eqn:oseenproblem}
	\begin{split}
		\min_{\vec{v},\,p,\,\vec{u}}\,
		& \frac{1}{2} \normb{\vec{v} - \widehat{\vec{v}}}_{L^2(\varOmega)^2}^2
		+ \frac{\alpha}{2} \normb{p - \widehat{p}}_{L^2(\varOmega)}^2
		+ \frac{\beta}{2} \normb{\vec{u}}_{L^2(\partial \varOmega_{\text{in}})^2}^2 \\
		& \text{such that}
		\begin{mycases}
			\text{s.\,t.} \; - \nu \nabla^2 \vec{v} + \vec{v}_h \cdot \nabla \vec{v} + \nabla p & = \vec{0}
			\quad \text{in} \;\, \varOmega, \\
			\nabla \cdot \vec{v} & = 0 \quad \text{in} \;\, \varOmega, \\
			\vec{v} &= \vec{0} \quad \text{on} \;\,  \partial\varOmega_D, \\
			\nu \tfrac{\partial\vec{v}}{\partial n} - p \vec{n} & = \vec{u} \quad \text{on} \;\,  \partial\varOmega_{\text{in}}, \\
			\nu \tfrac{\partial\vec{v}}{\partial n} - p \vec{n} & = \vec{0} \quad \text{on} \;\,  \partial\varOmega_{\text{out}}.
		\end{mycases}
	\end{split}
\end{equation}

Pošta and Roubíček \cite{posta2007optimal} suggest augmenting the cost function of \eqref{eqn:oseenproblem} by the term 
$-\int_\varOmega(\vec{v}\cdot\nabla\vec{v}_h)\cdot\vec{\lambda}_h$, where $\vec{\lambda}_h$ denotes the continuous version of the Lagrange
multiplier $\bm{\uplambda}$ from the previous iteration\footnote{This is motivated by viewing the Picard iteration as an SQP
type iteration 
for \eqref{eqn:navierproblem}.}; this approach is used in \cite{pearson2015preconditioned}. Then convergence of a distributed control problem 
can be proved, see \cite{posta2007optimal} for details. The practical advantage lies in a reduced number of iterations for high 
Reynolds numbers. Our treatment is restricted to relatively low Reynolds numbers; we will find that the Picard iteration with the system 
\eqref{eqn:oseenproblem} gives satisfactory results in this case, thus we will not consider this augmentation.

The linear convection operator $\vec{v}_h\cdot\nabla\vec{v}$ is discretized in the two-dimensional finite element space by
$\bm{N} = \blkdiag(N,N)$ with $N=[\int_{\varOmega}\big(\vec{v}_h\cdot\nabla\varphi_j\big)\varphi_i]$. If we define the vector
convection-diffusion operator $\bm{F}=\nu\bm{A}+\bm{N}$, then the discretized Oseen equations are given by
\begin{equation} \label{eqn:discroseen}
	\begin{bmatrix}
		\bm{F} & B^T \\
		B & O
	\end{bmatrix}
	\begin{pmatrix}
		\mathbf{v} \\
		\mathbf{p}
	\end{pmatrix} =
	\begin{pmatrix}
		\widehat{\bm{Q}}\mathbf{u} \\
		\mathbf{0}
	\end{pmatrix}.
\end{equation}

The cost functional remains unchanged from \eqref{eqn:stokesproblem} to \eqref{eqn:oseenproblem}, hence the discrete cost function is also
given by \eqref{eqn:discrcost}. Applying the discretize-then-optimize approach\footnote{The alternative optimize-then-discretize method does 
not necessarily result in a symmetric matrix; see for example \cite{pearson2013fast} and \cite[Section~6.3]{rees2010preconditioning} for
the convection duffusion equation. We will not consider it here.} gives the Oseen KKT system
\begin{equation} \label{eqn:oseenkkt}
	\begin{bmatrix}
		\bm{Q}_{\vec{v}} & O & O & \bm{F}^T & B^T \\
		O & \alpha Q_p & O & B & O \\
		O & O & \beta\bm{Q}_{\vec{u}} & -\widehat{\bm{Q}}^T & O \\
		\bm{F} & B^T & -\widehat{\bm{Q}} & O & O \\
		B & O & O & O & O
	\end{bmatrix}
	\begin{pmatrix}
		\mathbf{v} \\
		\mathbf{p} \\
		\mathbf{u} \\
		\bm{\uplambda} \\
		\bm{\upmu}
	\end{pmatrix} =
	\begin{pmatrix}
		\mathbf{b} \\
		\alpha\mathbf{d} \\
		\mathbf{0} \\
		\mathbf{0} \\
		\mathbf{0} \\
	\end{pmatrix}.
\end{equation}
\noindent Note that here $\bm{F}\neq\bm{F}^T$ since the convection matrix $\bm{N}$ is nonsymmetric.

\section{Preconditioning for the Navier--Stokes problem} \label{sec:navierprecond}

Due to the similarity of the KKT system for the Oseen problem to the one for the Stokes problem \eqref{eqn:stokeskkt}, our preconditioner
will be based on the Rees--Wathen method discussed in the previous section. Indeed, the (1,1) blocks of both systems are identical, namely 
the block-diagonal matrix ${\blkdiag(\bm{Q}_{\vec{v}},Q_p,\bm{Q}_{\vec{u}})}$, so there is no additional work needed here.

The main difference is the operator $\bm{F}$ instead of $\bm{A}$, which is increasingly nonsymmetric for higher Reynolds numbers. If we
use the Schur complement approximation dropping the low-rank perturbation $\tfrac{1}{\beta}\mathcal{L}$, namely
\begin{equation*}
	\mathcal{S} \approx \widetilde{\mathcal{S}} =
	\underbrace{\begin{bmatrix}
		\bm{F} & B^T\\
		B & O
	\end{bmatrix}}_{\eqqcolon \mathcal{K}}
	\begin{bmatrix}
		\bm{Q}_{\vec{v}}^{-1} & O \\
		O & \tfrac{1}{\alpha} Q_p^{-1} \\
	\end{bmatrix}
	\begin{bmatrix}
		\bm{F}^T & B^T \\
		B & O \\
	\end{bmatrix},
\end{equation*}
\noindent we need linear approximations for $\mathcal{K}$ and $\mathcal{K}^T$ that satisfy the Braess--Peisker conditions 
in Theorem~\ref{thm:braesspeisker}. Note that we cannot use Krylov subspace methods such as GMRES \cite{saad1986gmres} since they are 
nonlinear and hence unsuitable as preconditioners for MINRES. A candidate for such is the Uzawa type iteration for nonsymmetric systems 
introduced by Bramble et al. \cite{bramble1999uzawa}. It may be written in the same way as Algorithm~\ref{alg:uzawa}, the subtle difference 
here is however that $\widetilde{\bm{A}}$ is not a preconditioner for $\bm{F}$ but for the symmetric part
$\bm{F}_S=\tfrac{1}{2}(\bm{F}+\bm{F}^T)$ with eigenvalue bounds $1 \leq \lambda(\widetilde{\bm{A}}^{-1}\bm{F}_S) \leq \varDelta$. Similarly, 
$\widetilde{S}$ is a preconditioner for the symmetric part of the Schur complement $S_S=B\bm{F}_S^{-1}B^T$  with eigenvalue bounds
$\gamma \leq \lambda(\widetilde{S}^{-1}S_S) \leq 1$. Note that for the given approximations, the factor $1$ in both inequalities can be achieved 
by a scaling. Then Algorithm ~\ref{alg:uzawa} will converge if $\delta$ and $\tau$ are small enough; for a rigorous statement see 
\cite[Theorem~3.1]{bramble1999uzawa}.

There exist
efficient solvers for the Navier---Stokes equations such as the pressure convection-diffusion preconditioner 
\cite{silvester2001efficient,kay2002preconditioner} and the least-squares commutator preconditioner 
\cite{elman1999preconditioning,elman2006block}. However, these preconditioners do not satisfy the Braess--Peisker conditions and we do not 
know how they can be symmetrized efficiently to embed them in a Uzawa type iteration, thus we will not use them here.

The properties of $\mathcal{K}$ depend on the discrete convection operator $\bm{N}$, and the lower the viscosity parameter $\nu$, the
more so. Therefore we will study some spectral properties of $\bm{N}$ in the next subsection, and then construct a preconditioner
taking them into account.

\subsection{Matrix properties} \label{subsec:matrix}

Due to the block-diagonal structure of the vector-convection operator $\bm{N}$ we can restrict our discussion to the scalar case without 
loss of generality.

We consider a continuous bilinear form
\begin{equation} \label{eqn:bilinear}
	c(u,v) \coloneqq \int_{\varOmega} \big( \vec{w} \cdot \nabla u \big) v,
\end{equation}
\noindent where $\vec{w}$ is a given wind with $\nabla\cdot\vec{w}=0$; this is no restriction because in the Oseen system the wind is the 
solution of a (Navier--)Stokes boundary value problem. The bilinear form  $c(\,\cdot,\cdot\,)$ is associated with the convection operator which 
is usually thought of as being skew-selfadjoint. However, in general there will be some selfadjoint perturbation. We want to find an 
explicit expression for it. Application of the divergence theorem to \eqref{eqn:bilinear} gives

\begin{equation*}
	\begin{split}
		c(u,v) = - \int_{\varOmega} (v \vec{w}) \cdot \nabla u
		& = - \int_{\varOmega} \nabla \cdot (v \vec{w}) \, u + \int_{\partial\varOmega_N} u v \, \vec{w} \cdot \vec{n} \\
		& = - \int_{\varOmega} \big( (v \nabla \cdot \vec{w}) u + (\vec{w} \cdot \nabla v) u \big)
		+ \int_{\partial\varOmega} u v \, \vec{w} \cdot \vec{n} \\
		& = - \int_{\varOmega} \big( \vec{w} \cdot \nabla v \big) u + \int_{\partial\varOmega_N} u v \, \vec{w} \cdot \vec{n},
	\end{split}
\end{equation*}
\noindent where we used the product rule and the fact that $\vec{w}$ is divergence-free. The last equality uses the assumption that
$\vec{w}\cdot\vec{n} = 0$ on $\partial\varOmega_D$, i.\,e. there is no in- or outflow on the Dirichlet boundary. This is consistent with 
our channel domain where the Dirichlet boundary represents the channel walls. Thus, the selfadjoint part of
$c(\,\cdot,\cdot\,)$ is given by
\begin{equation} \label{eqn:selfadjperturb}
	h(u,v) \coloneqq \frac{1}{2} \big( c(u,v) + c(v,u) \big) = \frac{1}{2} \int_{\partial\varOmega_N} u v \, \vec{w} \cdot \vec{n}.
\end{equation}

It is important to understand the consequences of \eqref{eqn:selfadjperturb} in our finite-element framework. Consider the basis $\{\varphi_j\}$ for the finite element space. Then we can express 
the symmetric part of the scalar convection matrix $N_S = \tfrac{1}{2}(N+N^T)$ in the form

\begin{equation*}
	N_S = \big[ n_{ij}^S \big],
	\qquad n_{ij}^S = \int_{\partial\varOmega_N} \varphi_i \varphi_j \, \vec{w} \cdot \vec{n},
\end{equation*}

\noindent and the selfadjoint part of $c(\cdot,\cdot)$, applied to a function $v_h\in\Span\{\varphi_j\}$, becomes

\begin{equation} \label{eqn:symmperturb}
	h(v_h,v_h) = \frac{1}{2} \int_{\partial\varOmega_N} v_h^2 \, \vec{w} \cdot \vec{n}
	= \frac{1}{2} \sum_{i=1}^n \sum_{j=1}^n 
	\bm{v}_i \bm{v}_j \int_{\partial\varOmega_N} \varphi_i \varphi_j \, \vec{w} \cdot \vec{n}.
	= \frac{1}{2} \mathbf{v}^T N_S \mathbf{v}.
\end{equation}

\noindent In the discrete bilinear form $\mathbf{v}^T N \mathbf{v}$ the contribution of $\mathbf{v}^T N_S \mathbf{v}$ will clearly
dominate if the coordinate vector $\mathbf{v}$ (and hence the associated function $v_h$) is relatively large (in modulus) on the Neumann
boundary and small everywhere else. From \eqref{eqn:symmperturb} it is clear that the sign of the contribution of
$\mathbf{v}^T N_S \mathbf{v}$ is determined by the sign of $\vec{w}\cdot\vec{n}$. Recall that $\vec{n}$ is the outward normal vector on 
$\partial\varOmega$, this motivates the definition of

\begin{align*}
	\text{the outflow boundary} \quad
	\partial\varOmega_+ & \coloneqq \big\{ x \in \partial\varOmega_N \, \big| \, \vec{w} \cdot \vec{n} > 0 \big\}, \\
	\text{the characteristic boundary} \quad
	\partial\varOmega_0 & \coloneqq \big\{ x \in \partial\varOmega_N \, \big| \, \vec{w} \cdot \vec{n} = 0 \big\}, \\
	\text{and the inflow boundary} \quad
	\partial\varOmega_- & \coloneqq \big\{ x \in \partial\varOmega_N \, \big| \, \vec{w} \cdot \vec{n} < 0 \big\}.
\end{align*}

\begin{figure}[t!]
	\centering
	\includegraphics[scale=0.2]{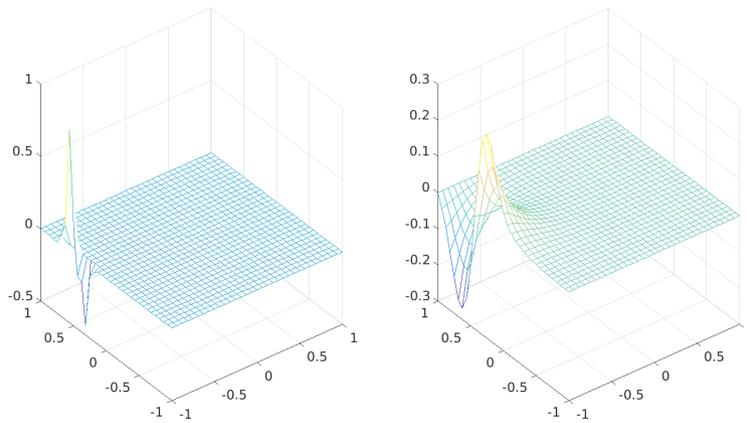}
	\caption{Eigenvectors of $N_S$ corresponding to $\lambda = -0.0314$ (left) and of ${F_S = N_S + \tfrac{\nu}{2}A}$ 
	corresponding to ${\lambda = -0.0020}$ (right) with $\nu = 1/10$.} 
	\label{fig:negeigenvalues}
\end{figure}

\noindent Note that these definitions do not necessarily coincide with the inflow $\partial\varOmega_{\text{in}}$ and the outflow
$\partial\varOmega_{\text{in}}$ of the channel; the latter are concepts based only on the geometry of the channel, while the former depend
on the actual flow profile. The quadratic form for $\mathbf{v}$ in \eqref{eqn:symmperturb} is a scaled Rayleigh quotient (especially, it has 
the same sign as the Rayleigh quotient), and thus, it is related to the eigenvalues of $N_S$. Clearly, the contribution of 
$\partial\varOmega_+$ is related to positive eigenvalues of $N_S$, and the contribution of $\partial\varOmega_-$ is related to negative 
eigenvalues.

We will concentrate on $\partial\varOmega_-$ and the associated negative eigenvalues of $N_S$. In our control problem the wind is the solution
in the previous iterate $\vec{v}_h$. Since we are controlling the inflow of the channel, we may assume that $\vec{v}_h$ is close enough to
the desired flow profile $\widehat{\vec{v}}$ such that ${\vec{v}_h\cdot\vec{n}<0}$ if and only if ${\widehat{\vec{v}}\cdot\vec{n}<0}$.
Thus, it is sufficient to consider the inflow boundary with respect to $\widehat{\vec{v}}$, which is known from \eqref{eqn:vhat}. Hence
the inflow boundary is $\partial\varOmega_- = \{-1\}\times(0,1)$.

This issue is illustrated in Figure~\ref{fig:negeigenvalues}. The eigenvector of $N_S$ corresponding to a negative eigenvalue lives mostly on the
inflow boundary; the eigenvector of the scalar convection-diffusion operator ${F_S \coloneqq N_S + \tfrac{\nu}{2}A}$ tends to
have greater components in the interior of $\varOmega$ but shows similar asymptotic behaviour. In general, for decreasing $\nu$ we may
expect that the eigenvalues and eigenvectors of $F_S$ converge to those of $N_S$.

\subsection{Permutational preconditioner} \label{subsec:permute}

As discussed in the introduction to this section, the Bramble--Pasciak--Vassilev Uzawa type iteration is guaranteed to converge only if the 
symmetric part of the (1,1) block of the saddle point system is positive definite. As we have seen in the previous subsection, this is in 
general not the case for the discretized Oseen operator in our boundary control problem.

We have also seen both a theoretical explanation and numerical evidence that the negative eigenvalues of the symmetric part of the
discrete convection-diffusion operator $F_S$ are associated with modes that live mostly on the inflow boundary and are close to zero 
everywhere else. Let $v_h$ be such an eigenmode of $F_S$. It has a representation in the finite element basis in the form
\begin{equation*}
	v_h = \sum_{j=1}^n \bm{v}_j \varphi_j.
\end{equation*}
\noindent Since we are employing Lagrange elements, each of the basis functions $\varphi_j$ is equal to $1$ at one node and equal to $0$ at the others. This means that
the coefficients $\bm{v}_j$ corresponding to basis functions that are zero on the inflow boundary will be relatively small in modulus;
thus, in the matrix-vector product $F_S\mathbf{v}$ those columns of $F_S$ will dominate which correspond to the inflow nodes. On the other
hand, the matrix $F$ is a weighted sum of the discrete Laplacian $A$ and the convection operator $N$; the former is well-known to be
positive-semidefinite, the latter is skew-selfadjoint with the exception of the boundary contribution \eqref{eqn:selfadjperturb} discussed 
in the previous subsection. Thus, if we remove the columns (and the corresponding rows, to retain symmetry) from $F_S$ corresponding to the 
inflow components, we will eliminate most of the contribution of the negative modes and might expect the resulting matrix to be positive 
definite.

We explore this issue for the viscosity parameter $\nu=1/20$ and grid size $h=10^{-5}$ in Table~\ref{tab:deflevs}. The matrices $F_S'$, 
$F_S''$ and $F_S^{(4)}$ represent $F_S$ with every inflow node, every second inflow node and every fourth inflow node eliminated, respectively, in the way 
described above. Clearly, in this case it is sufficient to eliminate every fourth node to get a positive definite matrix, however the 
condition number is increased in this case. Presumably, the elimination of a higher number of inflow nodes eliminates not only negative 
eigenvalues but also some positive ones that are close to zero, giving a better condition number.

\begin{table}[t!]
	\centering
  	\caption{Eigenvalue structures and condition numbers of different node elimination schemes.}
  	\label{tab:deflevs}
  	\begin{tabular}{|c||c|c|c|c|}
  		\hline
  		 & $F_S$ & $F_S'$ & $F_S''$ & $F_S^{(4)}$ \\ \hline\hline
  		\#(negative eigenvalues) & $2$ & $0$ & $0$ & $0$ \\ \hline
  		$\kappa$ & $2.5261\cdot 10^3$ & $1.8061\cdot 10^3$ & $1.8178\cdot 10^3$ & $4.9132\cdot 10^3$ \\ \hline
  	\end{tabular}
\end{table}

When eliminating the negative eigenvalues, it is desirable to preserve as much of the structure of $F_S$ as possible, i.\,e. we want the
positive eigenvalues to remain largely unchanged. Figure~\ref{fig:deflevs} shows evidence that this is, indeed, the case. The eigenvalues
of $F_S$ and $F_S'$ are mostly indiscernible, we only observe a small perturbation for the smallest ones; this is consistent with the
theory, as the smallest eigenvalues of $F_S$ are presumably the ones most influenced by the boundary behaviour of the convection operator.
We may expect even better behaviour if the number of deleted nodes is chosen to be smaller.

\begin{figure}[t!]
	\centering
	\begin{tabular} {cc}
		\includegraphics[scale=0.105]{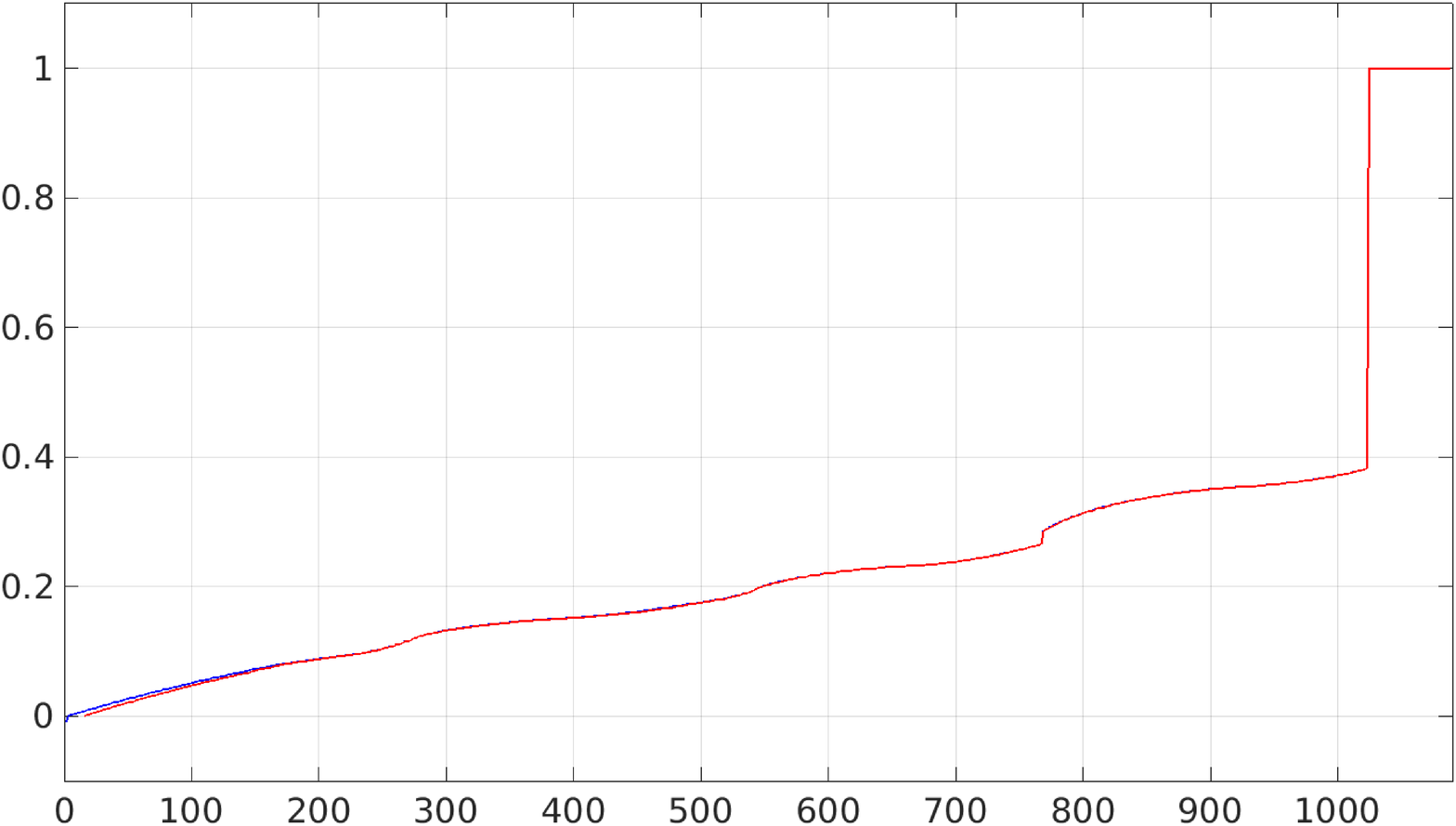} &
		\includegraphics[scale=0.105]{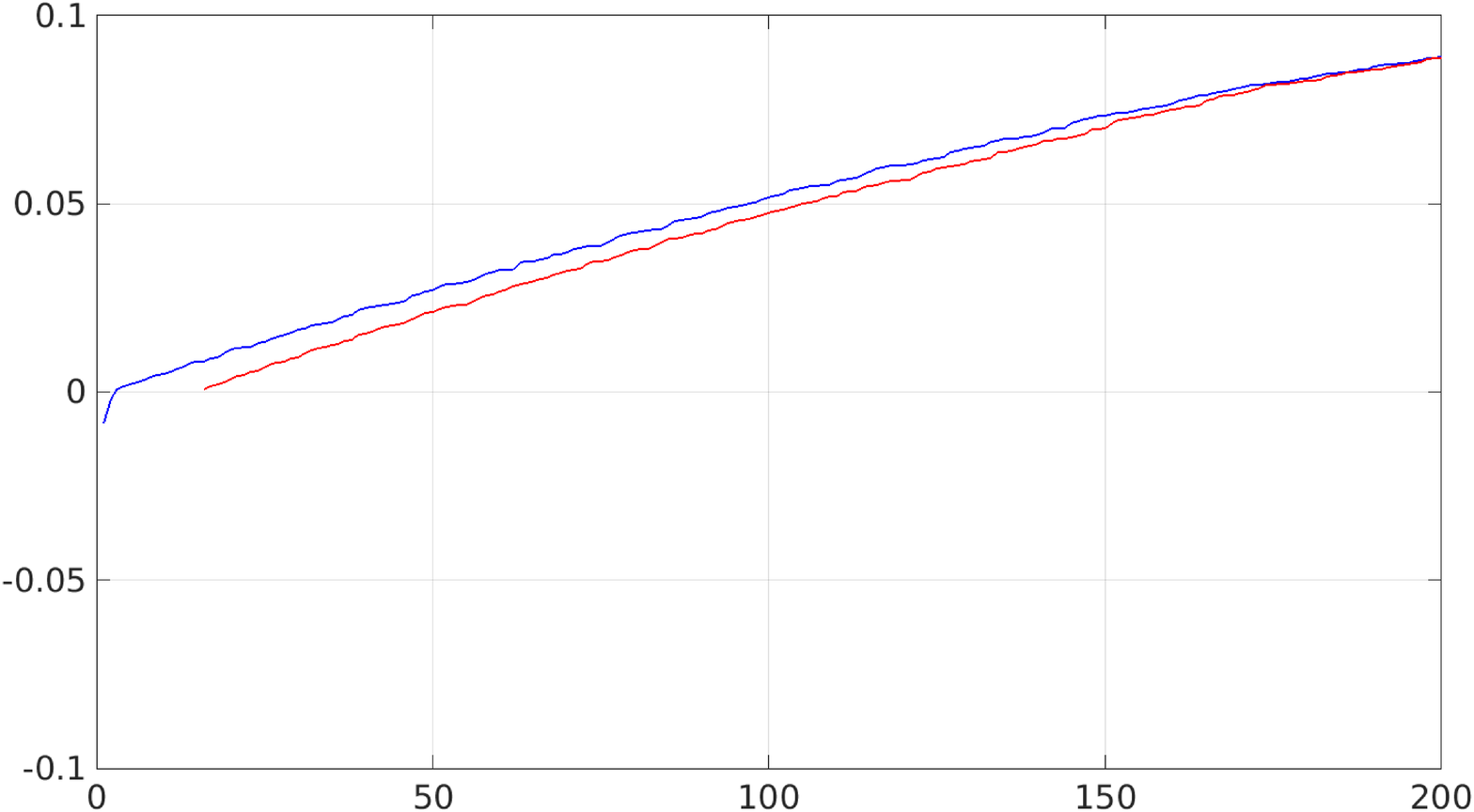} \\
	\end{tabular}
	\caption{Eigenvalues of $F_S$ (blue) and $F_S'$ (red) (plotted as continuous spectra), with a zoom in the lower left corner.}
	\label{fig:deflevs}
\end{figure}

Now we cannot simply delete rows and columns from the discrete Oseen operator as this would give us a different problem with a different
solution. Instead, we need to build a preconditioner that works around this issue. The idea of removing certain eigenspaces of a matrix
has been studied in the literature; a common technique is known as deflation. It is based on projecting the eigenspaces of the
system matrix into a space where the ``undesired'' eigenvalues are equal to zero---see, for example \cite{frank2001construction} for 
details. However, all these methods produce a singular system matrix, which is not a problem for Krylov subspace methods such as CG---they 
will still converge as long as the problem is consistent---but renders usual stationary iterations used in preconditioning infeasible. 
Therefore, deflation ideas cannot be used in our preconditioner in the usual way.

Instead, we will use an idea based on permutation. The negative eigenvalues of $F_S$ pose a problem inside the Uzawa type iteration, but if
we can move them to another part of our preconditioner, we might use the Uzawa type iteration for the remaining positive definite problem
and take care of the negative eigenvalues in some other way. We take the following approach: we change the order of rows and columns in
the system matrix of the discrete Oseen optimality conditions. If $\vec{\varphi}_j$ is a basis function which is equal to one at an inflow 
node we would like to remove it, so $\bm{a}_{jj}$ is moved to the (3,3) block of the KKT matrix. If the original system \eqref{eqn:oseenkkt} is given by
$\mathcal{A}\mathbf{x} = \mathbf{c}$, this can be thought of as solving the system
\begin{equation} \label{eqn:permute}
	\mathcal{P} \mathcal{A} \mathcal{P}^T \mathbf{y} = \mathcal{P} \mathbf{c}
\end{equation}
\noindent with a permutation matrix $\mathcal{P}$. The system \eqref{eqn:permute} is clearly symmetric and equivalent to the original
problem with $\mathbf{x} = \mathcal{P}^T\mathbf{y}$. This permutation of $\mathcal{A}$ is illustrated in Figure~\ref{fig:spy}.

\begin{figure}[t]
	\centering
	\includegraphics[scale=0.25]{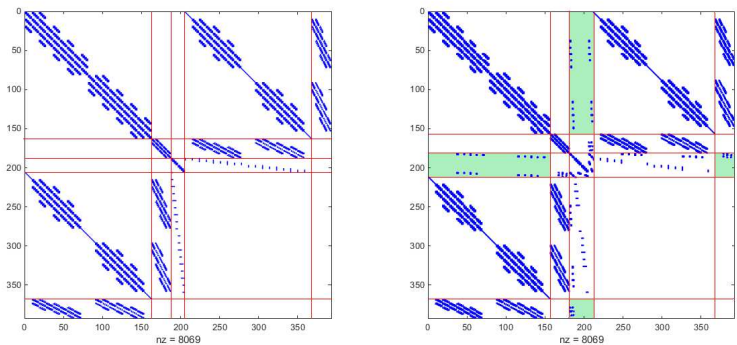}
	\caption{MATLAB $\mathtt{spy}$ plots of the Oseen KKT matrix without (left) and with (right) permutation. The blocks with
	extra fill introduced by the permutation are marked.}
	\label{fig:spy}
\end{figure}

Aside from augmenting the (3,3) block of the Oseen system, the permutation also changes the sprasity structure: it introduces some fill
in blocks that are equal to zero in the original system. We do not know how to deal with these efficiently, therefore in the computational 
preconditioner they are set to zero. This can be justified heuristically from two points of view. First, this is a variation of the idea 
that incomplete factorizations are based on: the preconditioner is enforced to have the same sparsity structure as the system matrix; the use of
incomplete factorizations is justified by both theoretical and numerical results. Second, the fill created by the permutation can be seen as 
a low-rank perturbation of a permuted system without fill; Theorem~\ref{thm:lowrank} suggests that we do not lose much information about
almost all eigenvalues by dropping it. Therefore, the spectrum of the system matrix remains largely unchanged, and we can expect a small number of additional iterations in the outer MINRES method.

Now, the Schur complement of the permuted system \eqref{eqn:permute} can be approximated in the same way as discussed in the previous
section: we  approximate the inverses of the Oseen operator and its transpose by inexact Uzawa type iterations and multiply by a
block-diagonal mass matrix, as discussed in Subsection~\ref{subsec:schur}. Since the permutation is designed to remove the nodes
associated with negative eigenvalues of $\bm{F}_S$, as discussed above, we are left with a positive-definite system and are guaranteed
to get convergence of the inner Uzawa type iteration.

The approximation of the block-diagonal upper-left block requires more care here, since the diagonal blocks are all not simple mass
matrices any more. While the (2,2) block, which is the pressure mass matrix, remains unchanged, by construction, the (3,3) block is
now much more complicated than just a simple mass matrix; recall that we use it to collect the boundary information corresponding to the
inflow, which we removed from the other blocks. We do not know enough about the structure of this matrix to find a good iterative preconditioner. However, especially for fine
grids, this will be a relatively small matrix, compared to the total number of degrees of freedom. On a two-dimensional domain, the number
of grid points (and thus, the dimension of the Oseen KKT matrix) is of order $\mathcal{O}(h^{-2})$. On the other hand the (3,3) block
only contains entries related to boundary nodes; their number is of order $\mathcal{O}(h^{-1})$. On a three-dimensional domain we would get
analogous behaviour with $\mathcal{O}(h^{-3})$ and $\mathcal{O}(h^{-2})$, respectively. This is illustrated in Table~\ref{tab:33block}. 

\begin{table}[t!]
	\centering
	\caption{Dimensions $N$ of the discrete Oseen KKT system and $n$ of the $(3,3)$-block of the permutational preconditioner with
	every second inflow node removed.}
  	\label{tab:33block}
  	\begin{tabular}{|c||c|c|c|c|c|c|}
    	\hline
    	$l$  & $2$ & $3$  & $4$   & $5$      & $6$  & $7$   \\ \hline\hline
    	$N$     & $128$    & $392$     & $1,352$    & $5,000$       & $19,208$   & $75,272$ \\ \hline    	
    	$n$   & $14$    & $26$     & $50$    & $98$       & $194$   & $386$ \\ \hline
  	\end{tabular}
\end{table}

\begin{figure}[b!]
	\centering
	\begin{tikzpicture}
		\node[draw, circle, very thick, red] (1) at (0,0)     {$1$};
		\node[draw, circle, very thick, red] (2) at (1.5,0)   {$2$};
		\node[draw, circle, very thick, red] (3) at (3,0)     {$3$};
		\node[draw, circle, very thick, red] (4) at (0,1.5)   {$4$};
		\node[draw, circle, very thick, red] (5) at (1.5,1.5) {$5$};
		\node[draw, circle, very thick, red] (6) at (3,1.5)   {$6$};
		\node[draw, circle, very thick, red] (7) at (0,3)     {$7$};
		\node[draw, circle, very thick, red] (8) at (1.5,3)   {$8$};
		\node[draw, circle, very thick, red] (9) at (3,3)     {$9$};
		\draw[very thick] (1) edge (2);
		\draw[very thick] (2) edge (3);
		\draw[very thick] (3) edge (6);
		\draw[very thick] (6) edge (9);
		\draw[very thick] (9) edge (8);
		\draw[very thick] (8) edge (7);
		\draw[very thick] (7) edge (4);
		\draw[very thick] (4) edge (1);
		\end{tikzpicture}
	\caption{Representation of $\pmb{Q}_1$ (left) and $\pmb{Q}_2$ (right) elements.} \label{fig:q2grid}
\end{figure}
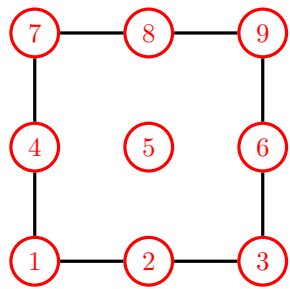

Note that the (3,3) block inherits the saddle point structure of the KKT matrix. Therefore, it is indefinite and cannot be used as a
preconditioner for MINRES. Instead, we use the block-diagonal preconditioner \eqref{eqn:bd}, which is also cheap due to the small dimension.

Now we consider the (1,1) block of the KKT system, namely the velocity mass matrix $\bm{Q}_{\vec{v}}$. Due to the block-diagonal
structure, it is sufficient to consider a scalar mass matrix $Q_v$. Recall that the permutation is constructed such that the contribution
of finite element basis functions on the inflow boundary is removed. We would like to precondition this modified mass matrix as usual
by a fixed number of Chebyshev semi-iterative steps. For this, we need explicit eigenvalue bounds of $D^{-1}Q$, where Q is the modified
mass matrix and $D=\diag(Q)$. This $Q$ is just a Galerkin mass matrix with the action of certain basis functions removed, thus the
eigenvalue results from \cite{wathen1987realistic} discussed earlier still apply here.

Especially, to get lower and upper bounds for the eigenvalues of $D^{-1}Q$, it is sufficient to compute the minimal and maximal eigenvalues
of $D_e^{-1}Q_e$ for all element mass matrices $Q_e$ and $D_e=\diag(Q_e)$. In each of the elements we remove the action of basis functions
that are equal to one at one edge node and equal to zero at all the others. This is equivalent to removing the row and column associated
with this node from a standard element matrix $Q_e$. We use a $\pmb{Q}_2$ approximation for the velocity space, represented in 
Figure~\ref{fig:q2grid}. A standard element mass matrix is a $9 \times 9$ matrix, since we have $9$ nodes in each element. We remove
some nodes located on one boundary, say the left-hand one. Then we need to consider index sets of the form $J \in 2^{\{1,\,4,\,7\}}$. We
get modified element mass matrices $Q_e^J$ by deleting the rows and columns of $Q_e$ with the indices in $J$. By the symmetry of the nodes 
at $1$ and $7$ we can restrict ourselves to the cases shown in Table~\ref{tab:permutationmass}. Clearly, the eigenvalues bounds from
Table~\ref{tab:massbounds} remain unchanged and we can apply the usual Chebyshev preconditioner.

\begin{table}[t]
	\centering
	\caption{Minimal and maximal eigenvalues of $(D_e^J)^{-1}Q_e^J$ for element mass matrices without nodes on the left-hand edge.}
  	\label{tab:permutationmass}
  	\begin{tabular}{|c||c|c|c|c|c|c|}
    	\hline
    	$J$                  & $\varnothing$ & $\{1\}$   & $\{4\}$    & $\{1,\,4\}$  & $\{1,\,7\}$  & $\{1,\,5,\,7\}$   \\ \hline\hline
    	$\lambda_{\min}$     & $0.2500$      & $0.3125$  & $0.3125$   & $0.3506$     & $0.3506$     & $0.3750$          \\ \hline    	
    	$\lambda_{\max}$     & $1.5625$      & $1.5625$  & $1.5625$   & $1.5625$     & $1.5625$     & $1.5625$          \\ \hline
  	\end{tabular}
\end{table}

\section{Numerical results} \label{sec:numres}

We present some numerical results for the problems \eqref{eqn:stokesproblem} and \eqref{eqn:navierproblem} using {MINRES} with the
Rees--Wathen type preconditioners discussed in the previous sections. All results were  obtained with MATLAB and the IFISS package
\cite{elman2007ifiss,elman2014ifiss}. The convergence condition for both MINRES and the Picard outer iteration is a residual reduction
by a factor of $10^{-6}$.

\subsection{Stokes control} \label{subsec:stokes}

We apply the Rees--Wathen type preconditioner from Section~\ref{sec:stokesprecond} with 5 Uzawa iterations as a Braess--Peisker 
approximation for the Stokes operator. The Laplacian is approximated by 5 AMG V-Cycles of the $\mathtt{HSL\_MI20}$ solver 
\cite{boyle2010hslmi20} which uses the Ruge--Stüben heuristics \cite{ruge1987algebraic}. The mass matrices are approximated with 20
steps of Chebyshev semi-iteration based on a damped Jacobi method. 

Typical solutions of the test problem are shown in Figure~\ref{fig:stokessolutions}. It is clearly seen how the regularization parameter 
$\beta$ influences the solution: for a small $\beta$ the optimal solution is closer to the desired velocity $\widehat{\vec{v}}$. For a small 
$\beta$ the action of the control is penalized less, and so we can influence the solution more by applying the optimal control. However, as 
we can see in Figure~\ref{fig:stokescontrol}, this effect appears to be bounded: $\mathbf{u}^T \bm{Q}_{\vec{u}} \mathbf{u}$, which is the 
discrete approximation of the control penalty function $\norm{\vec{u}}_{L^2(\partial\varOmega)}^2$, grows as $\beta$ decreases. However, the 
relatively small increase in the penalty function from $\beta = 10^{-5}$ to $\beta = 10^{-6}$ indicates that it may be bounded above by some 
constant. In the setting of the control problem, this would mean that there is a solution of the Stokes equations on this domain which 
minimizes the cost function with $\beta = 0$, possibly with a non-continuous control as the inflow boundary condition.

\begin{figure}[!]
	\centering
	\includegraphics[scale=0.2]{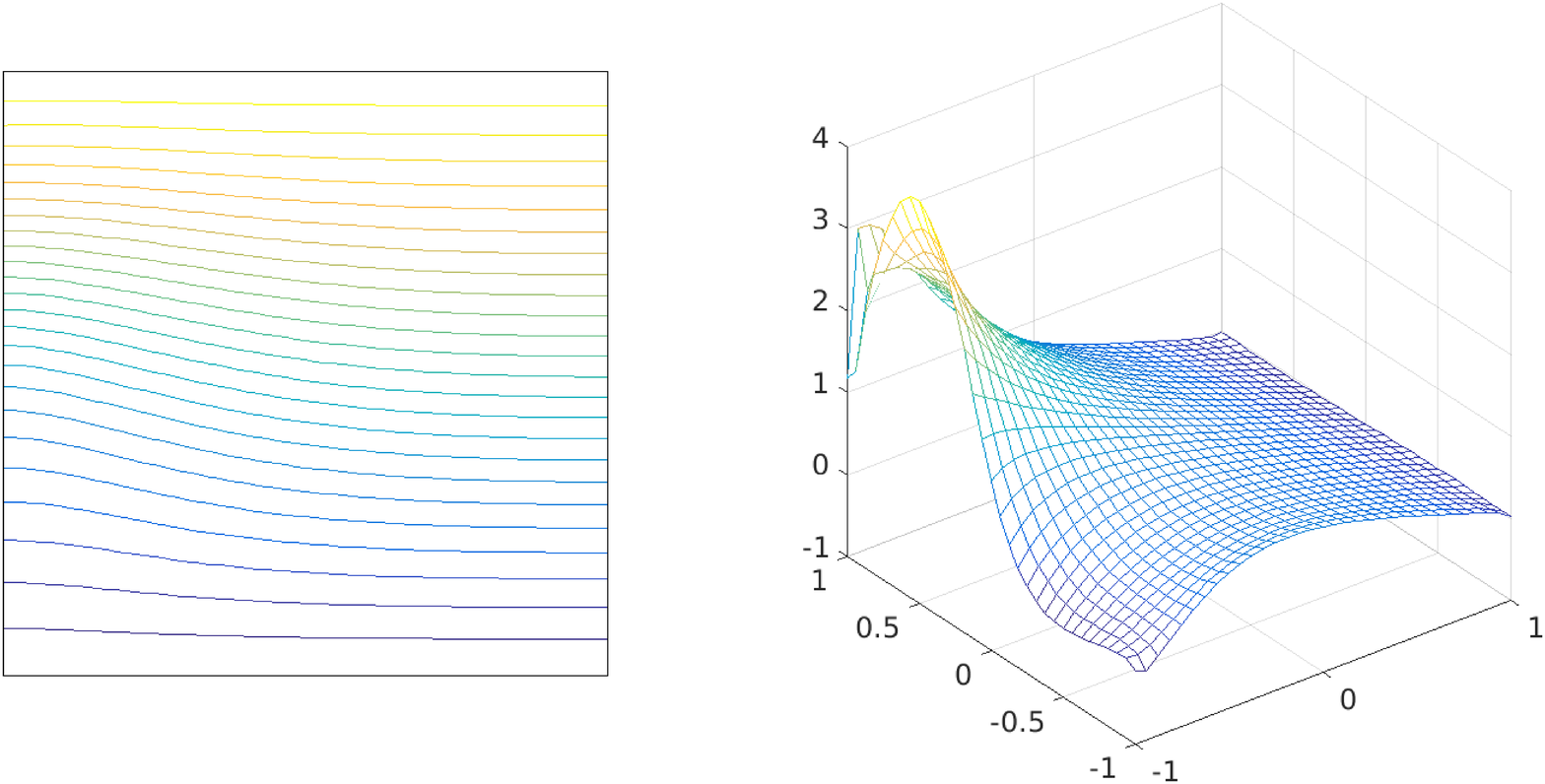} \\
	\includegraphics[scale=0.2]{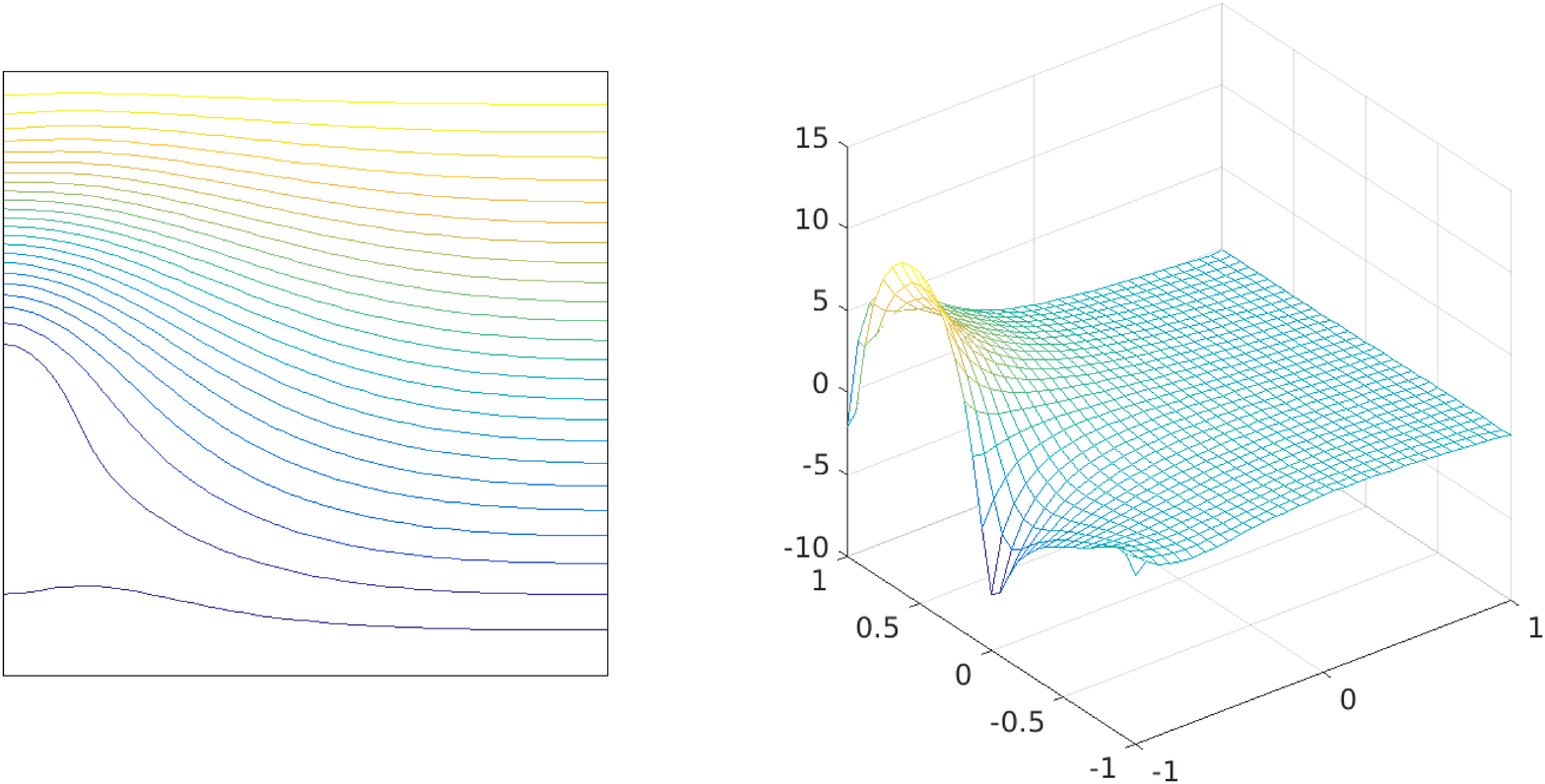}
	\caption{Solution of the test problem with $\alpha = 10^{-3}$ and $\beta = 10^{-2}$ (top) and $\beta=10^{-5}$ (bottom).} 
	\label{fig:stokessolutions}
\end{figure}

\begin{figure}[h!]
	\centering
	\begin{tabular} {cc}
		\includegraphics[scale=0.105]{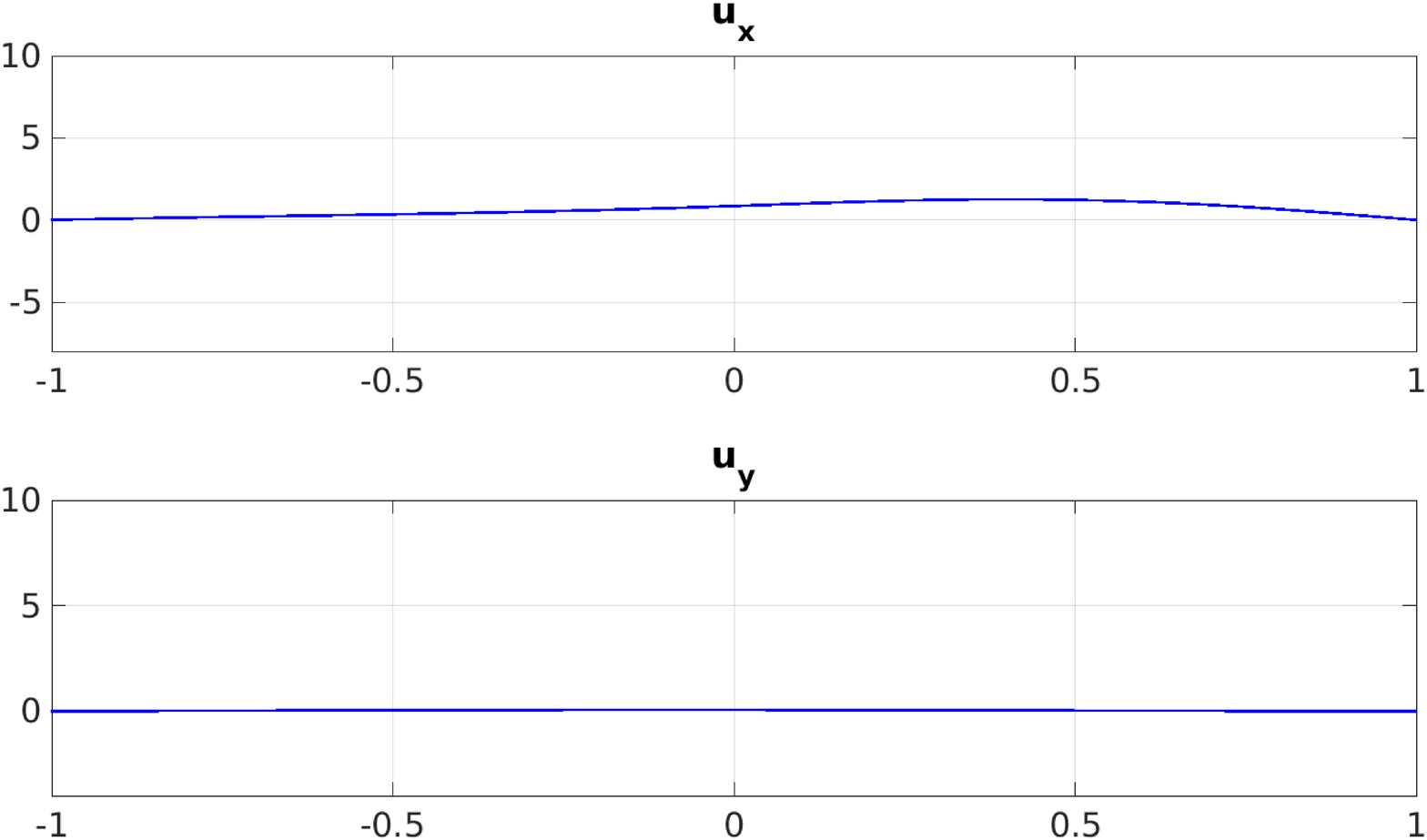} &
		\includegraphics[scale=0.105]{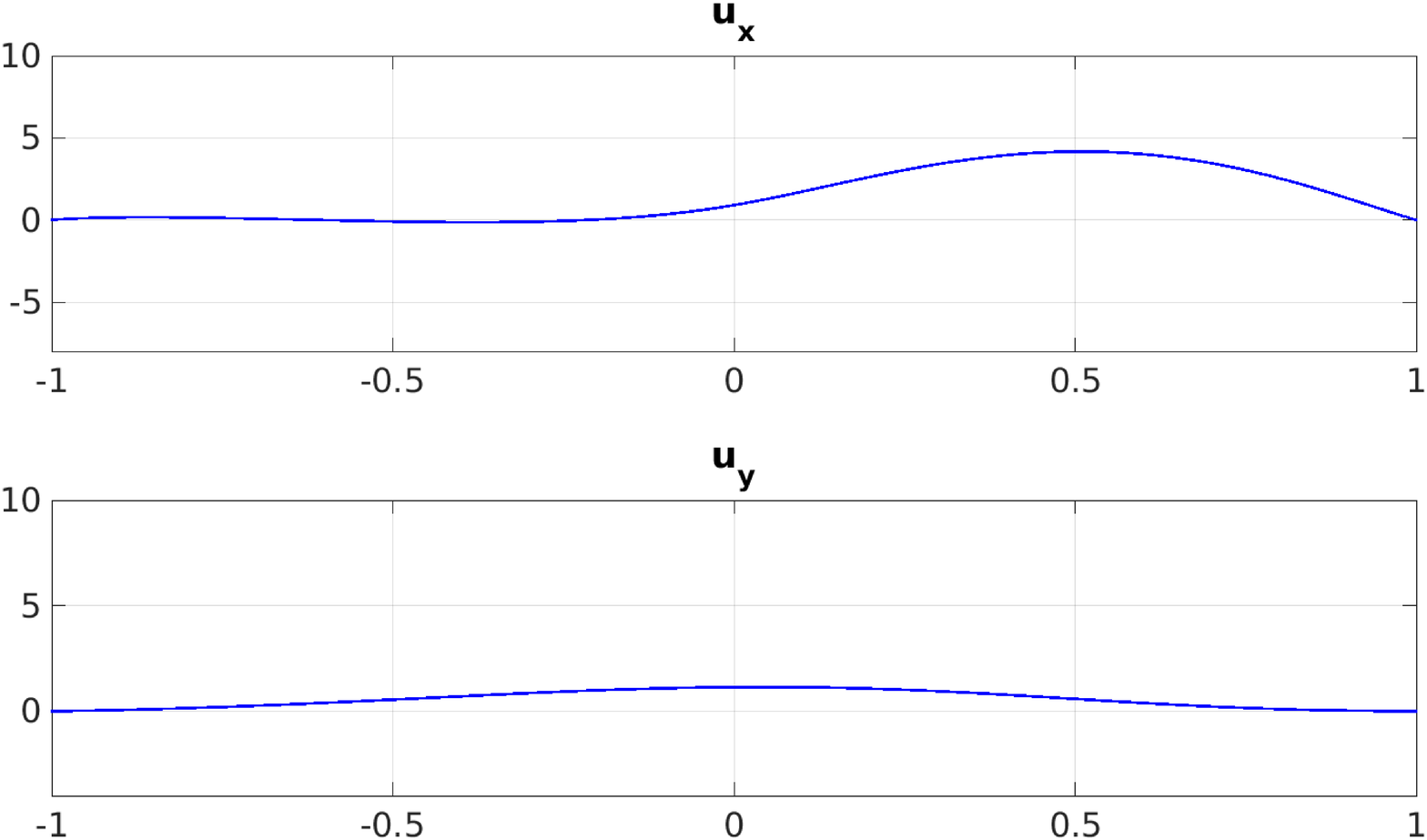} \\
		$\beta=10^{-1}$, $\mathbf{u}^T \bm{Q}_{\vec{u}} \mathbf{u} = 1.1677$ &
		$\beta=10^{-2}$, $\mathbf{u}^T \bm{Q}_{\vec{u}} \mathbf{u} = 10.1306$ \\
		\includegraphics[scale=0.105]{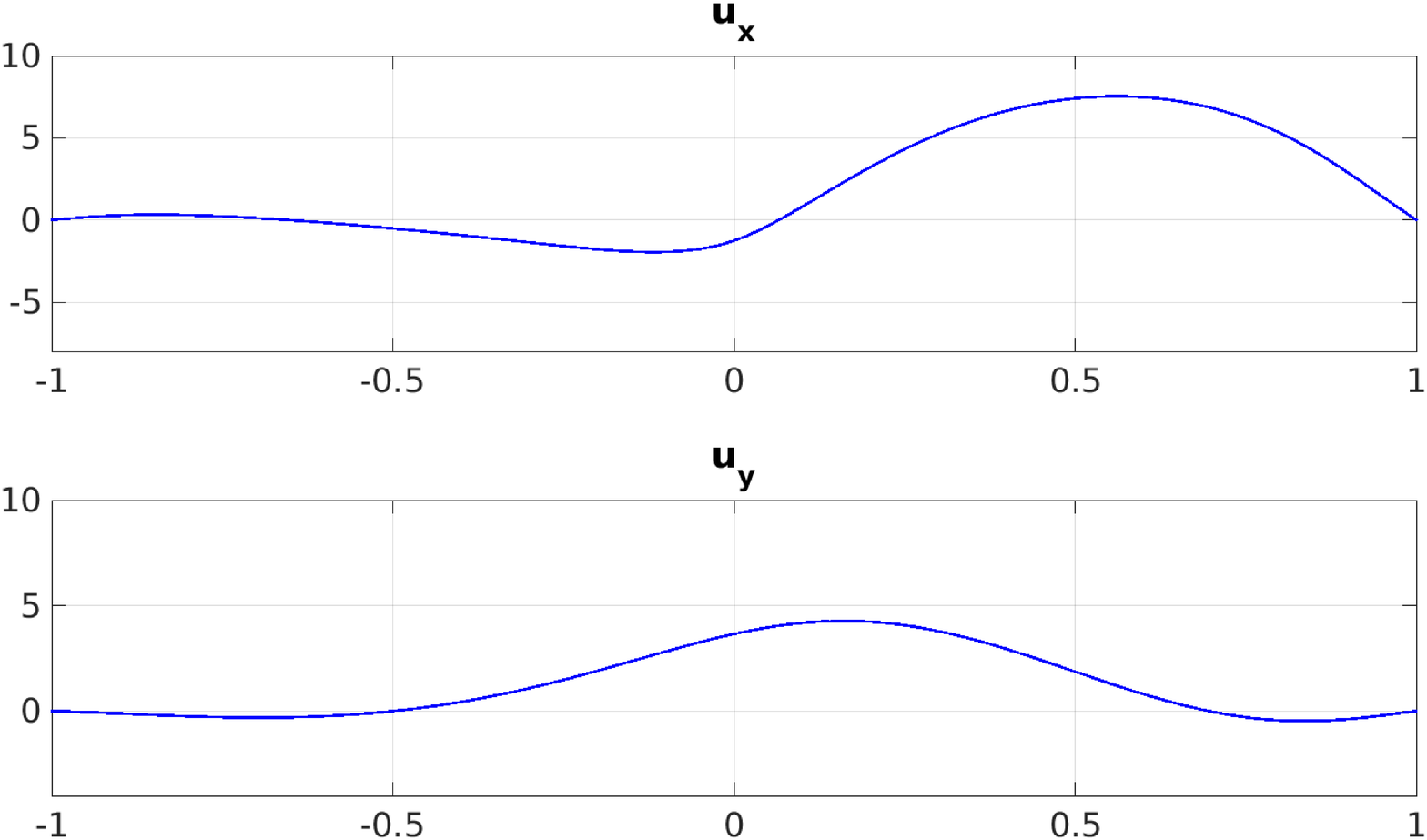} &
		\includegraphics[scale=0.105]{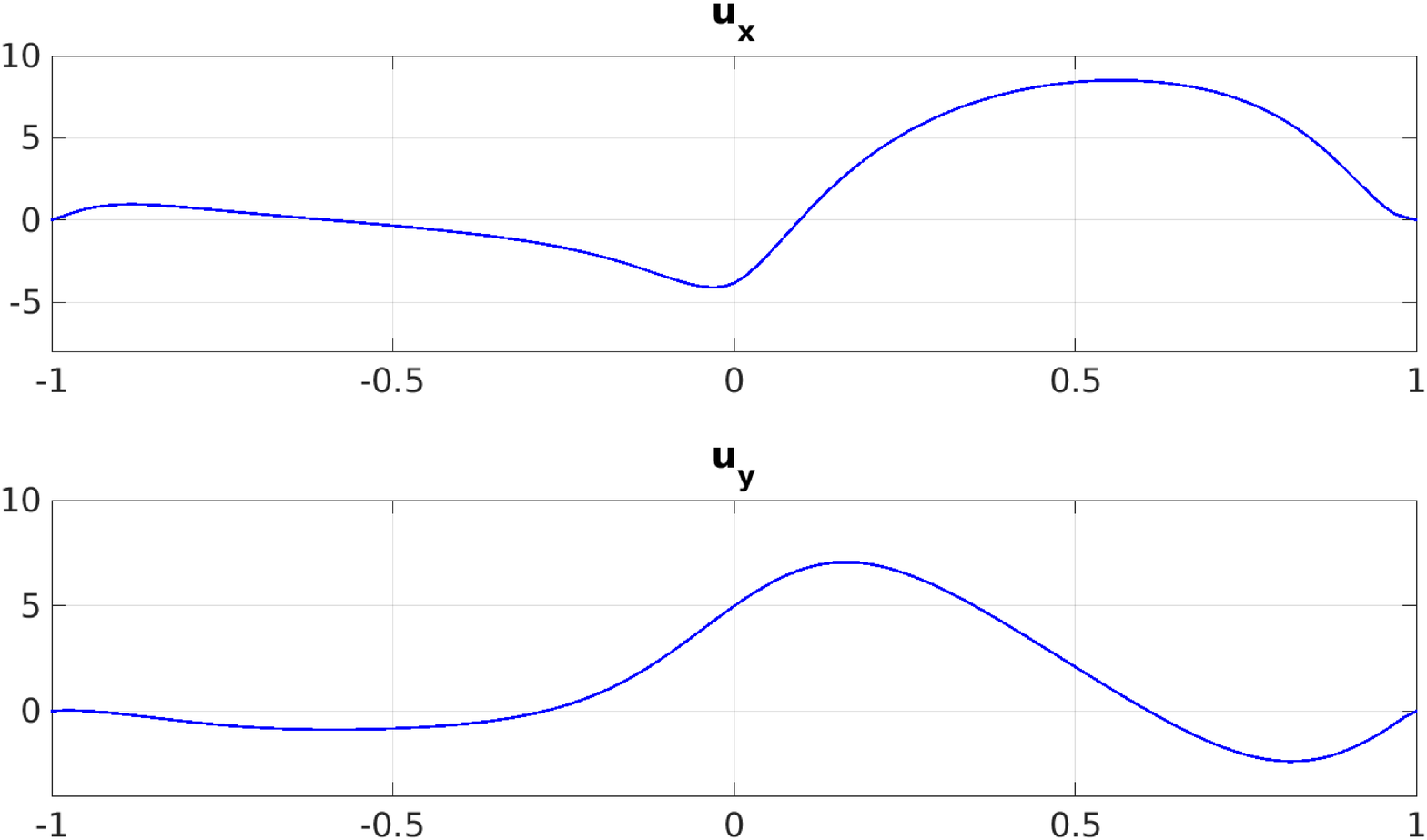} \\
		$\beta=10^{-3}$, $\mathbf{u}^T \bm{Q}_{\vec{u}} \mathbf{u} = 37.7790$ &
		$\beta=10^{-4}$, $\mathbf{u}^T \bm{Q}_{\vec{u}} \mathbf{u} = 59.6496$ \\
		\includegraphics[scale=0.105]{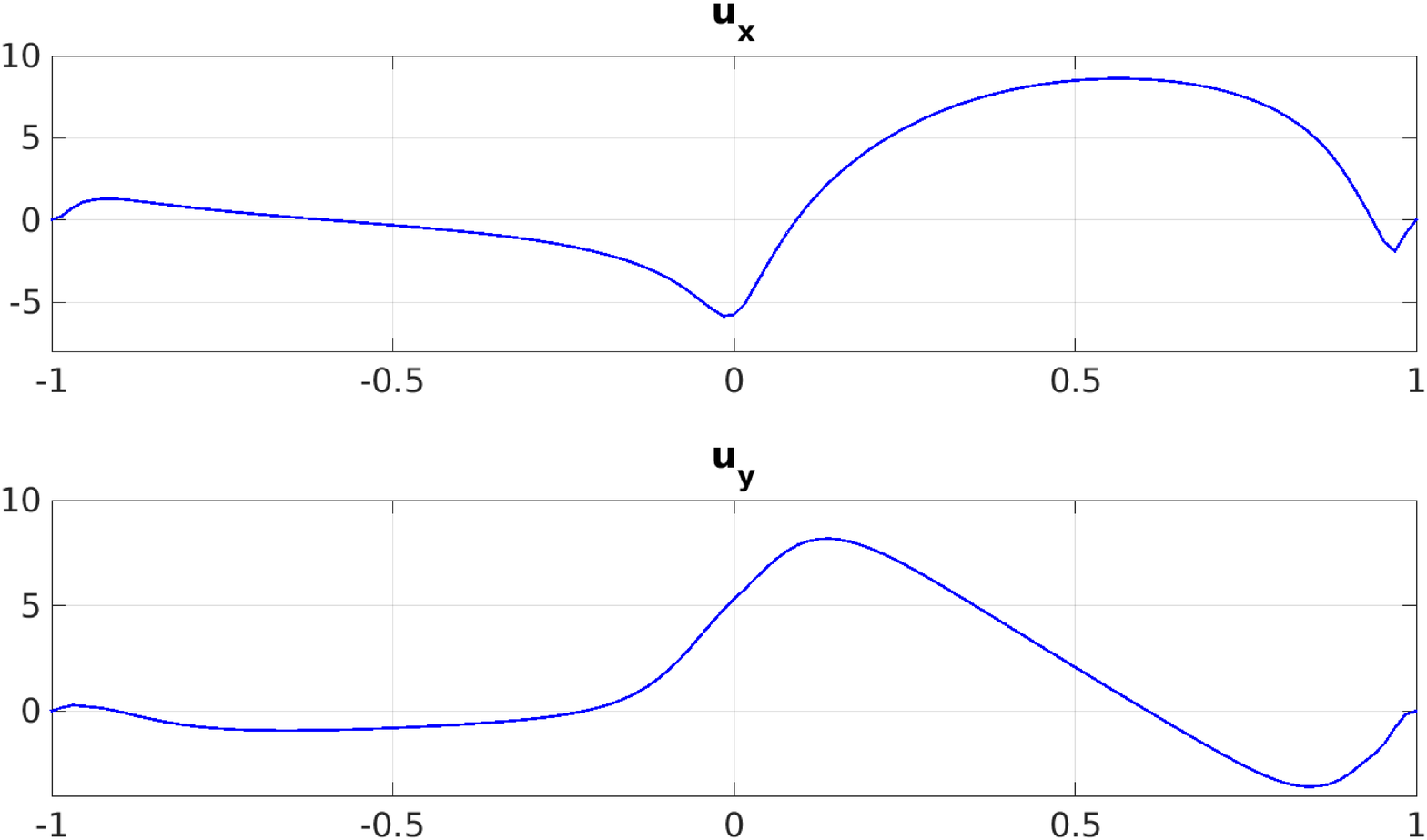} &
		\includegraphics[scale=0.105]{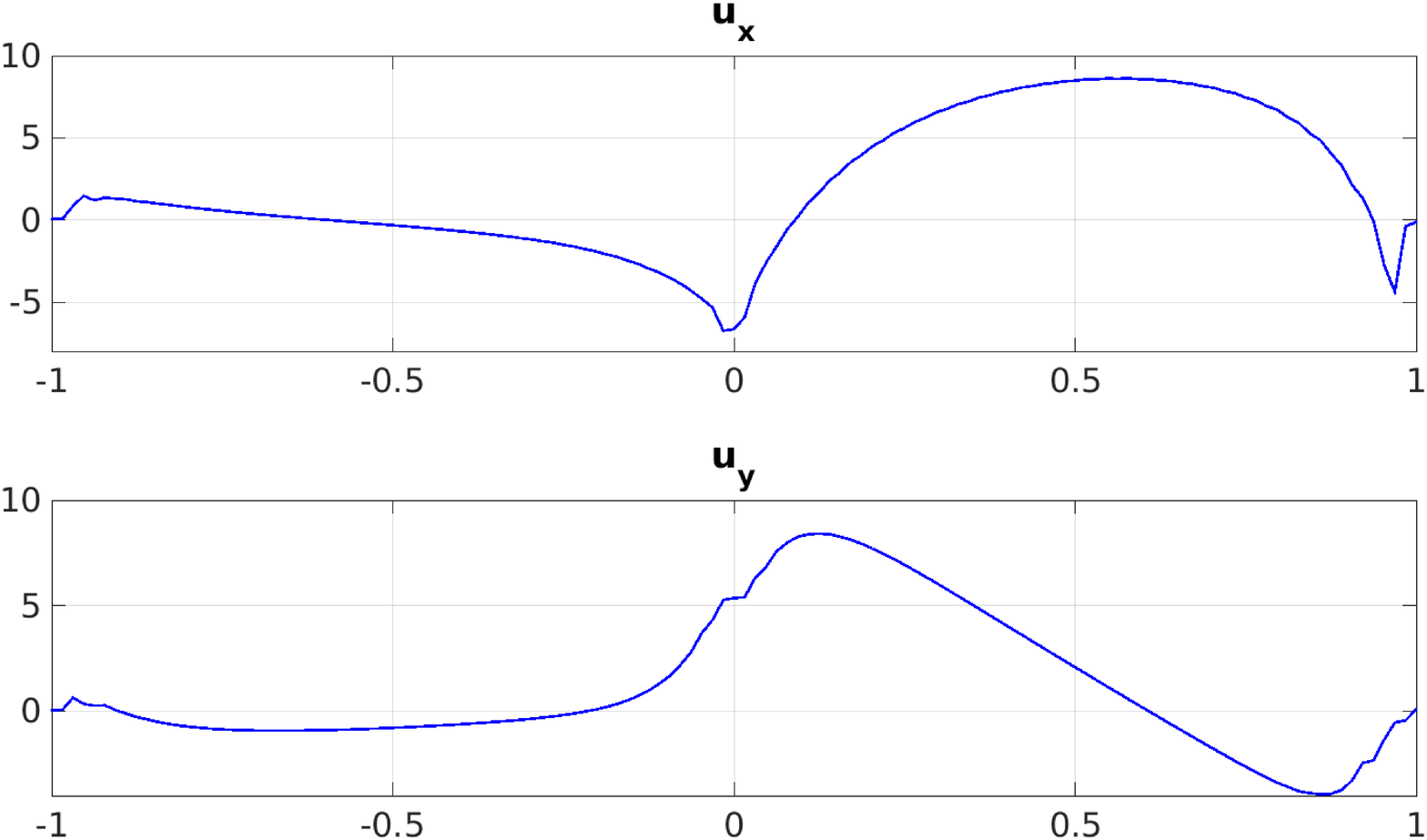} \\
		$\beta=10^{-5}$, $\mathbf{u}^T \bm{Q}_{\vec{u}} \mathbf{u} = 67.1488$ &
		$\beta=10^{-6}$, $\mathbf{u}^T \bm{Q}_{\vec{u}} \mathbf{u} = 68.8499$
	\end{tabular}
	\caption{Computed control with $\alpha = 10^{-3}$ and different $\beta$.} 
	\label{fig:stokescontrol}
\end{figure}

\begin{figure}[!]
	\centering
	\includegraphics[scale=0.2]{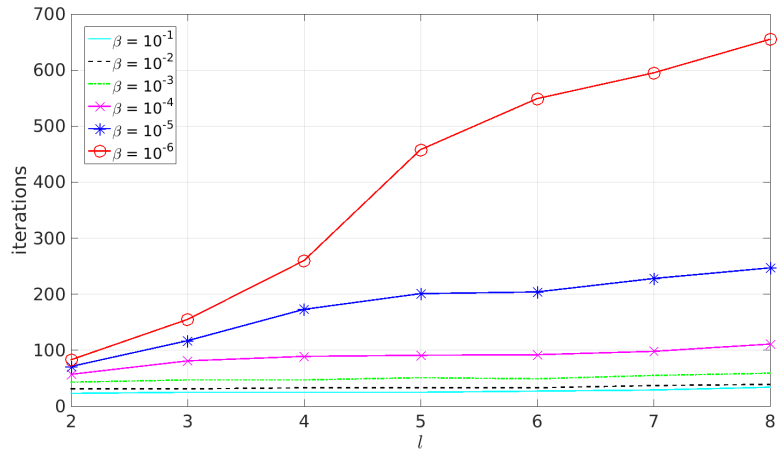}
	\caption{MINRES iteration counts with $\delta=10^{-3}$ and different $\beta$.} 
	\label{fig:differentbeta}
\end{figure}

The presumed upper bound for $\norm{\vec{u}}_{L^2(\partial\varOmega)}^2$ corresponds to a lower bound for 
$\norm{\vec{v}}_{L^2(\varOmega)}^2$. As can be seen in Figure~\ref{fig:stokessolutions}, even for relatively small $\beta$ we can only 
control the part of the channel close to the inflow; as it approaches the outflow, the flow profile resembles more and more the
Poiseuille flow we have seen in Section~\ref{sec:stokesproblem}. This coincides with our intuition: if the control can only be applied on 
some part of the boundary, its impact gets smaller the farther we move away from this boundary in the domain.

The preconditioning qualities are shown in Table~\ref{tab:precond}. The preconditioned MINRES iteration needs  
$\mathcal{O}(1)$ iterations and has a complexity of $\mathcal{O}(N)$ per iteration since only sparse matrix-vector products are required. 
Thus, the time required for the computation is linear in $N$. For relatively small matrices it is nevertheless outperformed by the MATLAB direct 
solver because of the constants involved. However, for very large matrices it is clearly better than the direct solver which has much higher 
memory requirements, which is why it achieves no solution for very fine grids. The MINRES method only needs to store vectors. The growth
rates of the computational time also suggest that MINRES should be faster for $h \leq 2^{-8}$ even with sufficient storage possibilities for $\varOmega \subset \R^2$.

\begin{table}[b!]
	\centering
	\caption{MINRES and MATLAB \texttt{backslash} performance on a $2^{-l}$ grid with $\alpha=\beta=10^{-3}$; the dash `---' indicates
	failure of the direct method.}
  	\label{tab:precond}
  	\begin{tabular}{|c||c|c|c|c|c|c|c|c|}
    	\hline
    	$l$                     & $2$   & $3$   & $4$     & $5$     & $6$        & $7$      & $8$       & $9$          \\ \hline\hline
    	DoF                     & $128$ & $392$ & $1,352$ & $5,000$ & $19,208$   & $75,272$ & $297,992$ & $1,185,800$  \\ \hline
    	MINRES iterations       & 43    & 48    & 47      & 54      & 54         & 57       & 61        & 63           \\ \hline
    	MINRES time             & 1.28  & 1.80  & 2.43    & 4.33    & 9.58       & 32.28    & 157.88    &  618.06      \\ \hline
    	\texttt{backslash} time & 0.01  & 0.01  & 0.04    & 0.2     &1.46        & 11.4     & ---       & ---          \\ \hline
  	\end{tabular}
\end{table}

The dependence of preconditioning qualities on the control regularization parameter $\beta$ is shown in Figure~\ref{fig:differentbeta}. We 
can observe the parameter-independence of the preconditioner for relatively large values of $\beta$. For the smallest values, especially 
$\beta=10^{-6}$, this behaviour deteriorates, but still the results compare favourably with the case of distributed control presented by 
Rees and Wathen \cite{rees2011preconditioning}.

\subsection{Navier--Stokes control} \label{subsec:navier}

\begin{figure}[b!]
	\centering
	\begin{tabular} {cc}
		\includegraphics[scale=0.105]{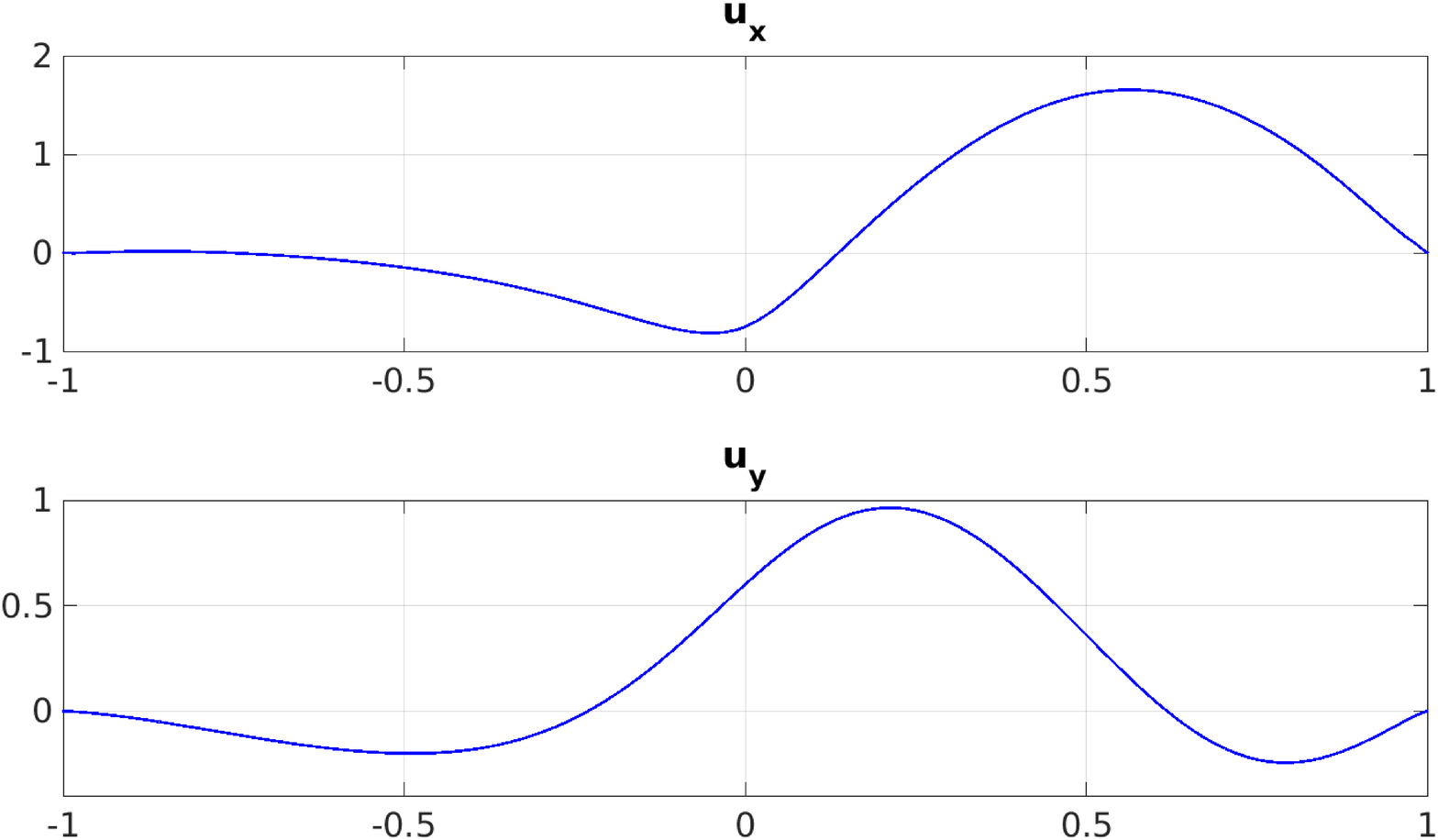} &
		\includegraphics[scale=0.105]{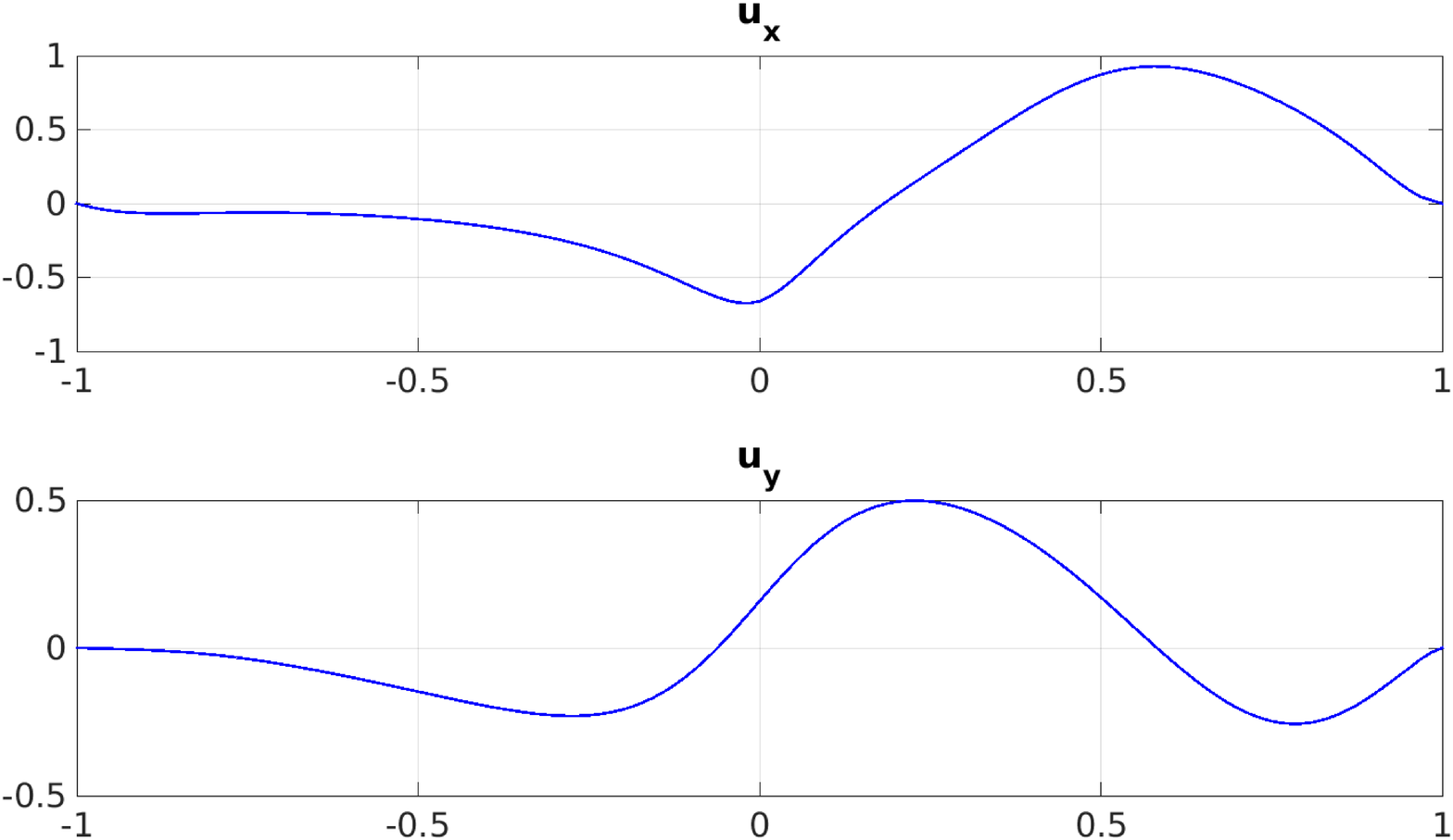} \\
		$\nu=1/5$, $\mathbf{u}^T \bm{Q}_{\vec{u}} \mathbf{u} = 1.7523$ &
		$\nu=1/10$, $\mathbf{u}^T \bm{Q}_{\vec{u}} \mathbf{u} = 0.5458$ \\
		\includegraphics[scale=0.105]{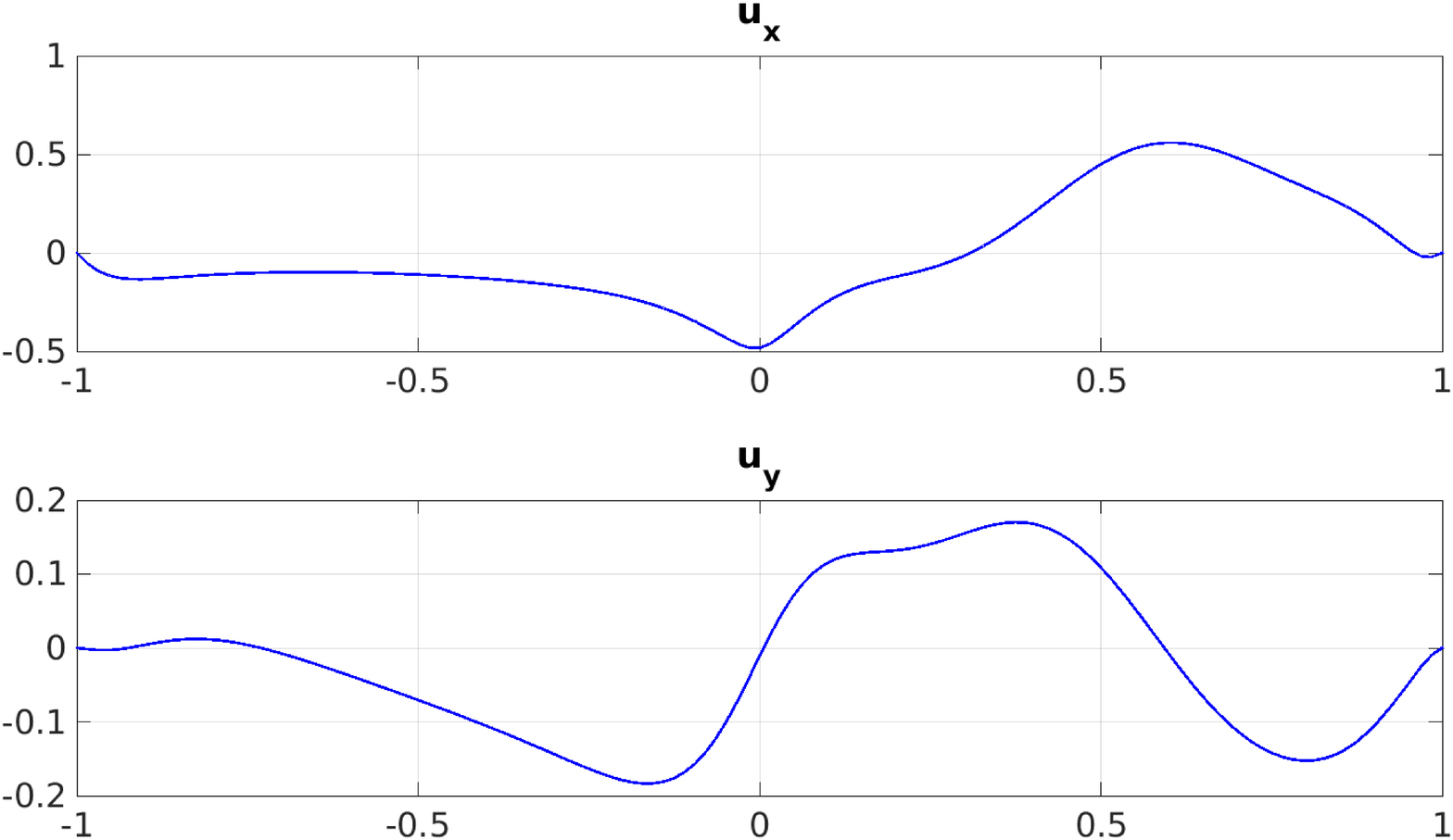} &
		\includegraphics[scale=0.105]{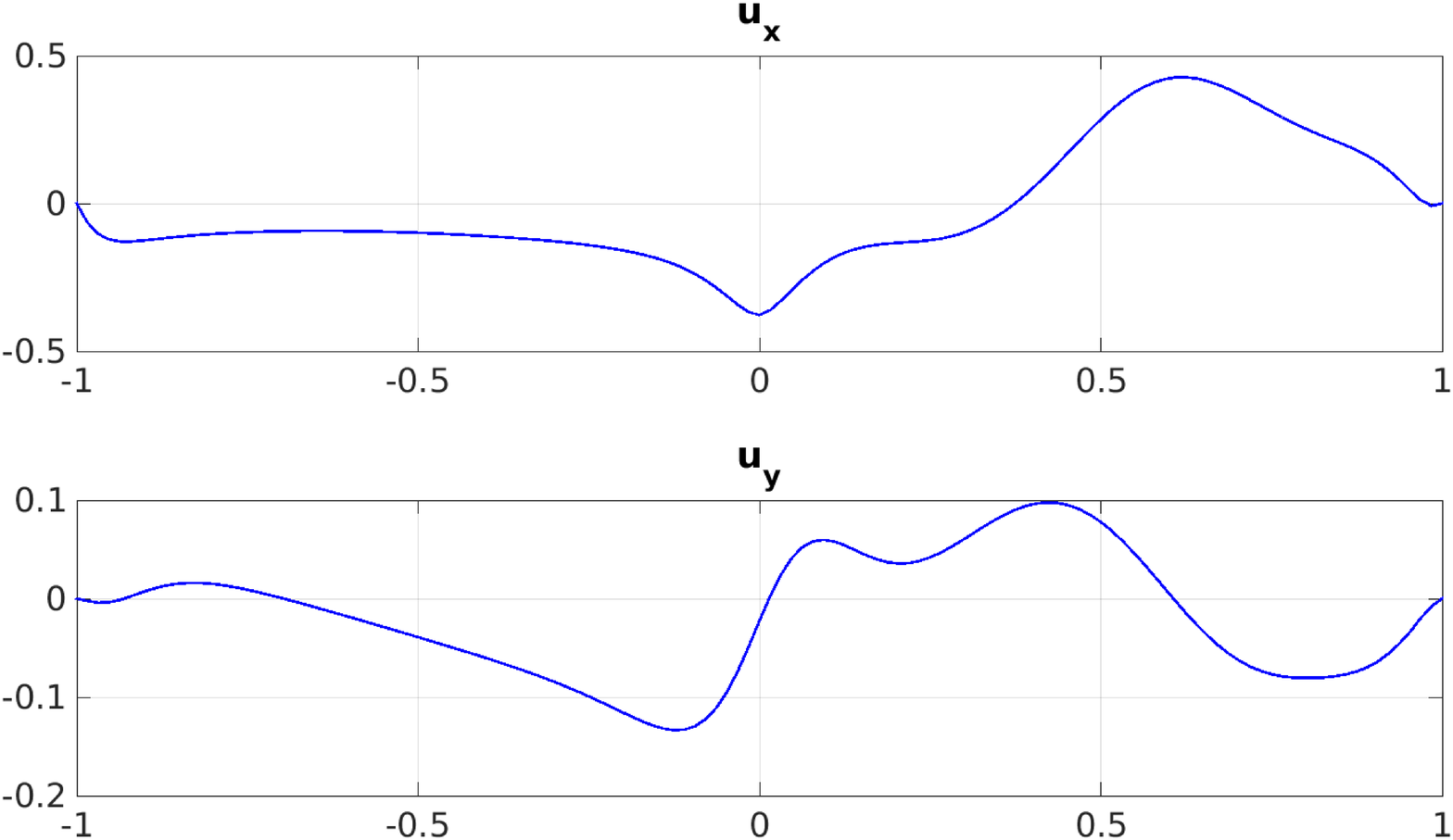} \\
		$\nu=1/20$, $\mathbf{u}^T \bm{Q}_{\vec{u}} \mathbf{u} = 0.1727$ &
		$\nu=1/30$, $\mathbf{u}^T \bm{Q}_{\vec{u}} \mathbf{u} = 0.0929$
	\end{tabular}
	\caption{Computed control with different viscosity parameters $\nu$.} 
	\label{fig:naviercontrol}
\end{figure}

We apply the permutational Rees--Wathen type preconditioner from Section~\ref{sec:navierprecond} with 30 Uzawa iterations as a Braess--Peisker approximation for the Oseen operator. Numerical experiments by Bramble et al. \cite[Section~6]{bramble1999uzawa} suggest using
a $\delta$ in Algorithm~\ref{alg:uzawa} depending on the viscosity parameter $\nu$; our choice is $\delta=\nu$. The parameter $\beta$ is 
chosen to be equal  to one. As a preconditioner for the (1,1) block in the discrete Oseen operator we chose a modified incomplete Cholesky
decomposition \cite{gustafsson1978class} implemented in MATLAB because the AMG solver used for the Stokes control problem in the previous 
chapter fails to construct meaningful algebraic course grid spaces here; we observe no sufficient smoothing. Note that incomplete 
decompositions in general do not give a grid-independent preconditioner; typical behaviour of the condition number of the preconditioned 
system is $\kappa=\mathcal{O}(h^{-1})$. The Schur complement approximation in the Uzawa type iteration is achieved, as usual, by a Chebyshev 
semi-iteration based on the pressure mass matrix.

\begin{figure}[h!]
	\centering
	\includegraphics[scale=0.2]{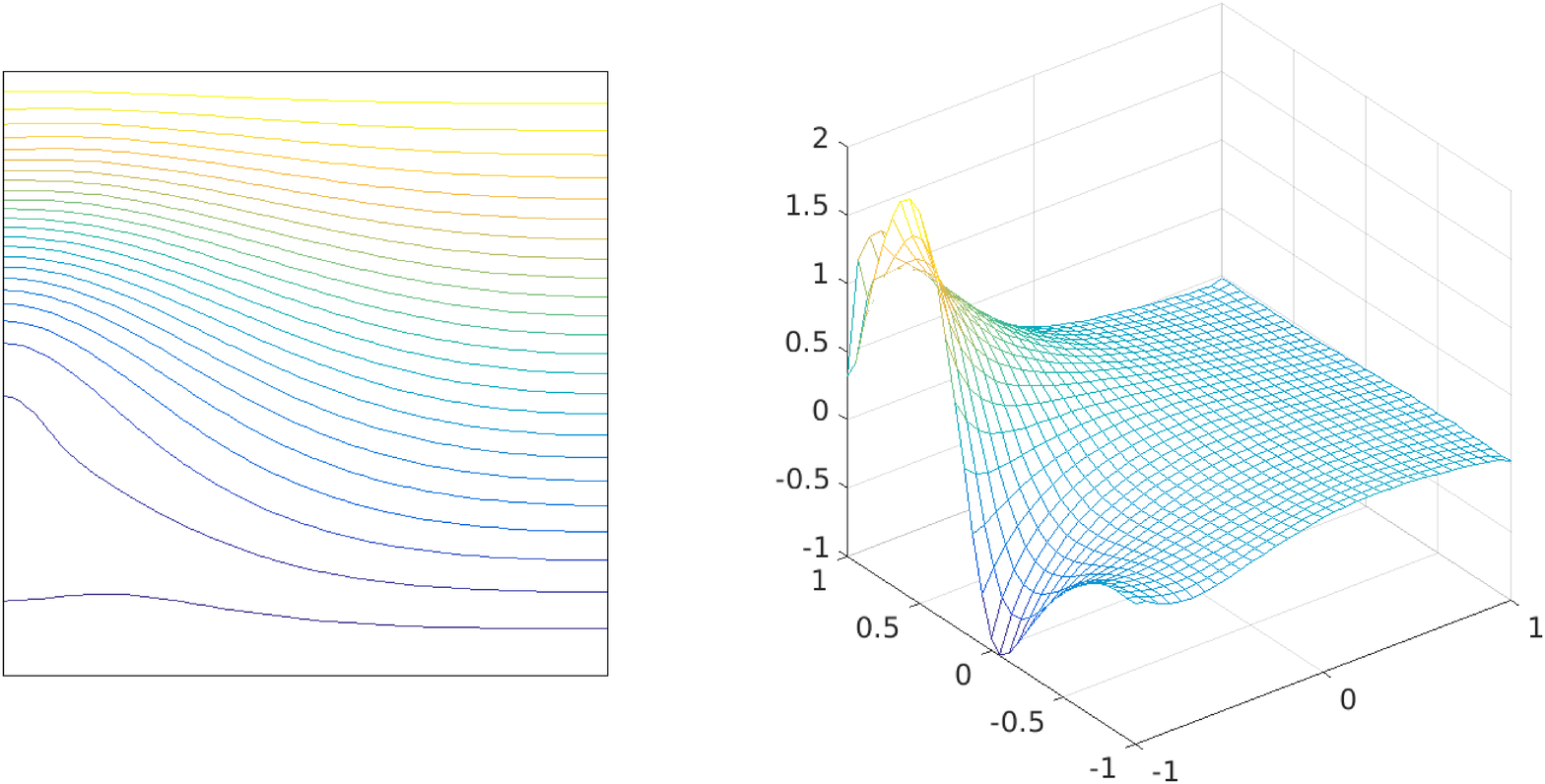} \\
	\includegraphics[scale=0.2]{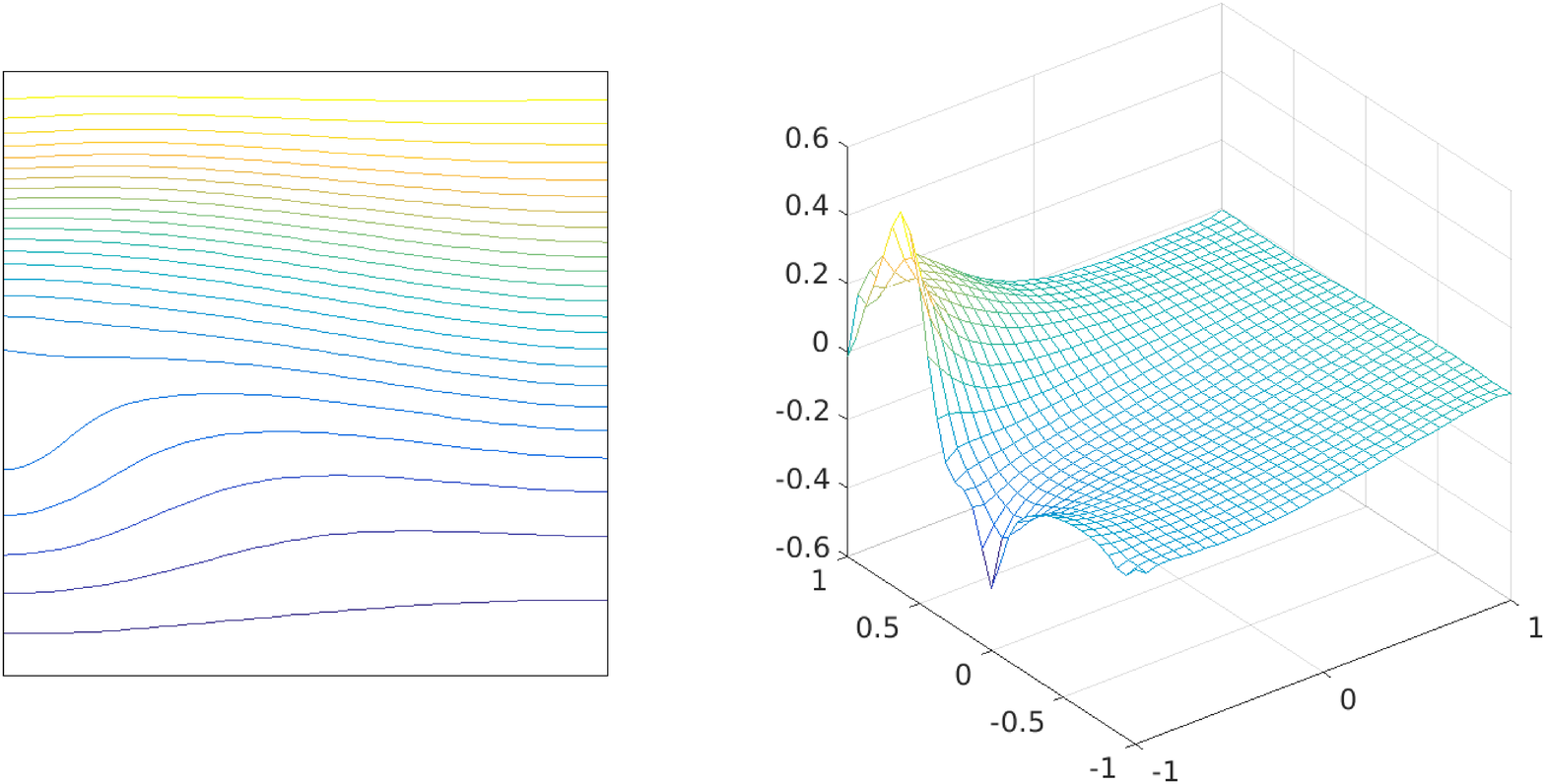}
	\caption{Solution of the test problem for $\nu = 1/5$ (top) and $\nu=1/30$ (bottom).} 
	\label{fig:naviersolutions}
\end{figure}

\begin{table}[!]
	\centering
	\caption{Average iteration counts and number of nonlinear iterations (in parentheses) for MINRES with the permutational preconditioner
	removing every, every second, every 4th, every 6th and every 8th inflow node (from top to bottom). The dash `---' indicates Uzawa 
	divergence due to indefiniteness.}
  	\label{tab:precondnodes}
  	\begin{tabular}{|c||x{1.7cm}|x{1.7cm}|x{1.7cm}|x{1.7cm}|x{1.7cm}|}
    	\hline
    	\diagbox{$\nu$}{$l$}   & $2$         & $3$           & $4$           & $5$            & $6$            \tabularnewline \hline\hline
    	$1/5$                  & $48\,(8)$   & $70\,(8)$     & $108\,(8)$    & $234\,(8)$     & $705\,(8)$     \tabularnewline \hline
    	$1/10$                 & $76\,(10)$  & $117\,(10)$   & $164\,(9)$    & $273\,(9)$     & $718\,(9)$     \tabularnewline \hline
    	$1/20$                 & ---         & $196\,(14)$   & $310\,(14)$   & $495\,(14)$    & $1,449\,(15)$  \tabularnewline \hline
    	$1/30$                 & ---         & $273\,(18)$   & $429\,(16)$   & $775\,(18)$    & $2,517\,(18)$  \tabularnewline \hline
  	\end{tabular} \\
  	\vspace{1cm}
  	\begin{tabular}{|c||x{1.7cm}|x{1.7cm}|x{1.7cm}|x{1.7cm}|x{1.7cm}|}
    	\hline
    	\diagbox{$\nu$}{$l$}   & $2$         & $3$           & $4$           & $5$            & $6$            \tabularnewline \hline\hline
    	$1/5$                  & $48\,(8)$   & $70\,(8)$     & $99\,(7)$     & $202\,(8)$     & $706\,(8)$     \tabularnewline \hline
    	$1/10$                 & $76\,(10)$  & $124\,(10)$   & $159\,(10)$   & $219\,(10)$    & $626\,(10)$    \tabularnewline \hline
    	$1/20$                 & ---         & ---           & $331\,(14)$   & $411\,(14)$    & $1,347\,(15)$  \tabularnewline \hline
    	$1/30$                 & ---         & ---           & $563\,(16)$   & $667\,(17)$    & $2,285\,(18)$  \tabularnewline \hline
  	\end{tabular} \\
  	\vspace{1cm}
  	\begin{tabular}{|c||x{1.7cm}|x{1.7cm}|x{1.7cm}|x{1.7cm}|x{1.7cm}|}
    	\hline
    	\diagbox{$\nu$}{$l$}   & $2$         & $3$           & $4$           & $5$            & $6$            \tabularnewline \hline\hline
    	$1/5$                  & ---         & $62\,(8)$     & $83\,(7)$     & $138\,(8)$     & $425\,(8)$     \tabularnewline \hline
    	$1/10$                 & ---         & $123\,(10)$   & $136\,(10)$   & $188\,(10)$    & $436\,(10)$    \tabularnewline \hline
    	$1/20$                 & ---         & ---           & ---           & $457\,(15)$    & $1,042\,(15)$  \tabularnewline \hline
    	$1/30$                 & ---         & ---           & $819\,(16)$   & ---            & $1,727\,(18)$  \tabularnewline \hline
  	\end{tabular} \\
  	\vspace{1cm}
  	\begin{tabular}{|c||x{1.7cm}|x{1.7cm}|x{1.7cm}|x{1.7cm}|x{1.7cm}|}
    	\hline
    	\diagbox{$\nu$}{$l$}   & $2$         & $3$           & $4$           & $5$            & $6$            \tabularnewline \hline\hline
    	$1/5$                  & ---         & $66\,(8)$     & $73\,(7)$     & $133\,(7)$     & $346\,(8)$     \tabularnewline \hline
    	$1/10$                 & ---         & ---           & ---           & $200\,(10)$    & $406\,(10)$    \tabularnewline \hline
    	$1/20$                 & ---         & $367\,(14)$   & $1,294\,(14)$ & ---            & $928\,(15)$    \tabularnewline \hline
    	$1/30$                 & ---         & ---           & ---           & ---            & $3,341\,(19)$  \tabularnewline \hline
  	\end{tabular} \\
  	\vspace{1cm}
  	\begin{tabular}{|c||x{1.7cm}|x{1.7cm}|x{1.7cm}|x{1.7cm}|x{1.7cm}|}
    	\hline
    	\diagbox{$\nu$}{$l$}   & $2$         & $3$           & $4$           & $5$            & $6$            \tabularnewline \hline\hline
    	$1/5$                  & ---         & ---           & $72\,(8)$     & $120\,(8)$     & $318\,(8)$     \tabularnewline \hline
    	$1/10$                 & ---         & ---           & ---           & $192\,(10)$    & $381\,(10)$    \tabularnewline \hline
    	$1/20$                 & ---         & ---           & ---           & ---            & $970\,(15)$    \tabularnewline \hline
    	$1/30$                 & ---         & ---           & ---           & ---            & ---            \tabularnewline \hline
  	\end{tabular} \\
\end{table}

A typical solution of the test problem is shown in Figure~\ref{fig:naviersolutions}. As we would expect, for a viscous, heavily
diffusion-dominated flow with $\nu=1/5$ the optimal solution is very similar to the Stokes case in Figure~\ref{fig:stokessolutions}. For a 
less viscous flow with $\nu=1/30$, the optimal solution is perceivably different and much closer to the desired flow profile. Also, the 
pressure difference needed to maintain the flow is smaller, which is consistent with the general theory for the Navier--Stokes equations.

Figure~\ref{fig:naviercontrol} shows the computed control for different values of $\nu$. The $\bm{Q}_{\vec{u}}$-norm of the control,
and hence the energy required to obtain the optimal solution, decreases with the viscosity parameter $\nu$. This means that for relatively
high Reynolds numbers not only is the optimal solution closer to the desired state but it is also (in a physical sense) cheaper to achieve.
Also, we observe that the qualitative behaviour of the control, i.\,e. the number and location of maxima and minima, is similar to the
Stokes case in Figure~\ref{fig:stokescontrol} for low Reynolds numbers. For higher Reynolds numbers, i.\,e. $\mathcal{R} \geq 20$, we
observe a qualitatively different behaviour.

The preconditioning qualities are shown in Table~\ref{tab:precondnodes}. In contrast to the Stokes case, we do not observe
grid-independent behaviour. This is related to two issues: first, we lose some of the structure of our problem by permuting the rows and
columns of the system matrix and dropping a low-rank perturbation; second, in contrast to AMG, the incomplete Cholesky preconditioner
applied here cannot guarantee grid-independent behaviour. Nevertheless, we do achieve useful preconditioning qualities, especially
in the case $\nu=1/10$. We can conclude that preconditioners of this kind can be useful for Navier-Stokes problems with relatively low
Reynolds numbers. For coarse grids, the removal of every inflow node is a good choice, for finer grids (which is probably more interesting 
in realistic applications) it appears advantageous to remove only some of the inflow nodes. However, if not enough inflow nodes are removed,
we risk getting an indefinite system, especially on relatively coarse grids.

\section{Conclusions and possible extensions} \label{sec:conclusions}

In this article, we extended the application of a preconditioner for distributed Stokes control problems presented in the literature to the 
case of boundary control. It speeds up the convergence of MINRES considerably; as long as the regularization parameter $\beta$ is not too 
small the convergence is independent of the grid size. But even for small values of $\beta$ we get useful preconditioning properties. We 
provide a theoretical explanation for this in the form that for low $\beta$ the preconditioning only deteriorates in terms of a low-rank 
perturbation. We believe that low-rank structures can be exploited in preconditioning for a range of optimal control problems, another 
example in recent literature is \cite{ye2013some}. An objective for future research might be the development of parameter-independent 
preconditioners for Stokes boundary control, in the sense introduced in \cite{pearson2012new}.

In a next step, we have discussed the applicability of preconditioners of this type to Navier--Stokes boundary control problems. We have
seen that applicability is limited due to the intrinsic structures of Navier--Stokes Neumann boundary value problems, namely the
indefiniteness of the symmetric perturbation of the convection operator. We have presented a way to deal with these issues for problems
governed by Navier--Stokes equations with low Reynolds numbers. Our heuristic strategy can be justified in terms of a low rank perturbation.
We have presented numerical results that support the theoretical reasoning. This preconditioner does not show grid-independent behaviour
because we do not have a grid-independent approximation for the non-standard finite element operator involved. We believe that the 
performance of this preconditioner can be greatly improved by the application of an appropriate multigrid solver. The development of such a 
solver would require a closer analysis of this operator and could be the subject of future work.

A natural extension of the framework presented here is time-dependent boundary 
control, this can be combined with the work done in 
\cite{stoll2013allatonce} for time-dependent distributed control. In our test 
problems the energy of the control is typically
low, in fact, there is numerical evidence that it might be bounded above 
independently of the regularization parameter $\beta$. However,
for other applications it is conceivable that control constraints might be 
useful; this would require an active-set strategy as implemented in 
\cite{stoll2012preconditioning}. State constraints might also be useful in some 
settings, we refer to recent work in
\cite{herzog2010preconditioned,pearson2014preconditioners}. Another issue is the 
potential non-smoothness and even non-continuity of the 
control; if this is undesired from a practical point of view, it could possibly 
be dealt with by introducing regularization in a Sobolev 
seminorn, see \cite{john2014optimal}; practical preconditioners for these 
problem formulations are yet to be developed.

\section*{Acknowledgements}

The first author has been supported by the German Academic Scholarship Foundation during his stay in Oxford.

\bibliographystyle{wileyj}
\bibliography{precond}

\end{document}